\documentclass[aps,pre,onecolumn,groupedaddress]{revtex4}

\usepackage[latin1]{inputenc}
\usepackage{amsmath}
\usepackage{amsfonts}
\usepackage{amssymb}
\usepackage{amsthm}
\usepackage[english]{babel}
\usepackage{graphicx}
\usepackage{hhline}
\usepackage{hyperref}
\usepackage{float}
\usepackage{mathrsfs}
\usepackage{color}

\newtheorem{lemma}{Lemma}[section]

\newtheorem{theorem}{Theorem}[section]
\newtheorem{remark}{Remark}[section]

\newcommand{\R}{\ensuremath{\mathbb{R}}}

\allowdisplaybreaks 

\begin{document}


\title{Optimal control for a SIR model with limited hospitalised patients}

\author{Roc\'io Balderrama}
\affiliation{IMAS (UBA-CONICET), Intendente G{\"u}iraldes 2160 - Ciudad Universitaria, 
            C1428EGA, Buenos Aires, CABA, Agentina}
\affiliation{Departamento de Matem{\'a}tica, FCEyN-UBA, Intendente G{\"u}iraldes 2160 - Ciudad Universitaria - Pabell{\'o}n~I, C1428EGA, Buenos Aires, CABA, Agentina}
       
\author{Mariana In\'es Prieto}
\affiliation{Departamento de Matem{\'a}tica, UNS and INMABB (UNS-CONICET), Av. Alem 1253, Bah{\'i}a Blanca, 8000, Prov. de Buenos Aires, Argentina}
    
\author{Constanza S\'anchez de la Vega}
\affiliation{Departamento de Matem{\'a}tica, FCEyN-UBA, Intendente G{\"u}iraldes 2160 - Ciudad Universitaria - Pabell{\'o}n~I, C1428EGA, Buenos Aires, CABA, Agentina}
\affiliation{Instituto de C\'alculo (UBA-CONICET), Intendente G{\"u}iraldes 2160 - Ciudad Universitaria, C1428EGA, Buenos Aires, CABA, Agentina}            

\author{Federico Vazquez}            
\affiliation{Instituto de C\'alculo (UBA-CONICET), Intendente G{\"u}iraldes 2160 - Ciudad Universitaria, C1428EGA, Buenos Aires, CABA, Agentina}

\begin{abstract}
This paper analyses the optimal control of infectious disease propagation using a classic susceptible-infected-recovered (SIR) model characterised by permanent immunity and the absence of available vaccines. The control is performed over a time-dependent mean reproduction number, in order to minimise the cumulative number of ever-infected individuals (recovered), under different constraints. We consider constraints on isolation measures ranging from partial lockdown to non-intervention, as well as the social and economic costs associated with such isolation, and the capacity limitations of intensive care units that limits the number of infected individuals to a maximum allowed value. We rigorously derive an optimal quarantine strategy based on necessary optimality conditions. The obtained optimal strategy is of a boundary-bang type, comprising three phases: an initial phase with no intervention, a second phase maintaining the infected population at its maximum possible value, and a final phase of partial lockdown applied over a single interval. The optimal policy is further refined by optimising the transition times between these phases. We show that these results are in excellent agreement with the numerical solution of the problem.
\end{abstract}

\maketitle





\section{Introduction}

The outbreak of infectious diseases has always been a challenge to public health systems and society at large. Among the mathematical models used to understand and manage the spread of such diseases, the Susceptible-Infectious-Recovered (SIR) model introduced by
Kermack  and McKendrick \cite{kermack1927contribution} has emerged as a fundamental framework. The SIR model provides valuable insights into the dynamics of epidemics, allowing us to quantify the impact of various control measures on disease transmission.

For many disease outbreaks when a vaccine is still not available, such as for Covid-19, different countries around the world had to take a series of unprecedented decisions as imposing  non-pharmaceutical interventions like social distance, quarantines and total lock-downs as the most effective tools to mitigate the spread of the disease. 
Although these kinds of measures are helpful in reducing the virus transmission and giving time to health systems to adapt, they could be extremely stressful in terms of economic and social costs, and, in longer periods, tend to have less compliance with the population.

In the context of managing infectious disease outbreaks, one critical concern is the pressure placed on healthcare systems, particularly the capacity of hospitals to provide care to those in need. During epidemics, the demand for hospitalisation can often exceed available resources, leading to a higher mortality rate and a significant burden on healthcare workers.

We consider the classical SIR model widely used in epidemiology 
\cite{anderson1992infectious,brauer2012mathematical}.  As it is usual in these models, we assume that people who have recovered develop immunity
and, therefore, would not be able to get infected nor infect others. 

To account for quarantine measures, we consider a time-dependent reproduction number and we assume that the intervention will occur in the preset period of time $[0,T]$, as proposed by Greenhalgh \cite{greenhalgh1988some} and more recently by Ketcheson \cite{ketcheson2020optimal}, Bliman et al \cite{Bliman2021} and Balderrama et al \cite{Balderrama22}, since it is unrealistic that interventions
can be sustained indefinitely.  

As mentioned in \cite{Bliman2021} (see references therein) there are other methods to control disease outbreaks in the framework of the SIR: pharmaceutical interventions as vaccination or treatment; or the screening and quarantining of infected. Also different modelling frameworks have been considered such as SEIR (with the addition of the Exposed compartment),  SIRQ (with the addition of the Quarantine compartment) or SIRH (with the addition of the Hospitalised compartment) among others.

In this article we address the development and analysis of an optimal control framework that seeks, by forcing social distancing, to strike a balance between minimising the {\bf final size of epidemic} while keeping both the maximum {\bf  economic cost} that the policy maker can afford and ongoing available {\bf hospital resources}. 
By incorporating constraints on the economic costs and the number of patients that can be admitted to hospitals, we aim to provide a realistic and actionable approach to epidemic management. 
The first constraint consists of an inequality for  $L^1-$norm on the control [see \eqref{eq: int_cond}], while the second constraint is a peak prevalence type of the infected for all time (maximum number allowed of infected), a running state constraint that considers the hospitalised patients as a fraction of the infected individuals.

This research contributes to the ongoing effort to enhance our understanding of disease dynamics and offers practical strategies for policymakers and healthcare professionals to effectively combat infectious disease outbreaks while operating within real-world economic and healthcare constraints.

Recently, many works have appeared dealing with these and related issues.
The problem of maximising the final number of susceptible individuals was studied by Ketcheson \cite{ketcheson2020optimal} in a prescribed intervention interval of time $[0,T]$. However, in that work it is assumed that the
total lockdown corresponds to a zero reproduction number, something that is impossible to achieve in the real world.
Moreover, it is assumed that the total lockdown can last during the whole intervention,
which seems to be impracticable. 
Later, in \cite{Bliman2021} the authors considered the same optimal control problem with a partial lockdown.
In these two articles it is proved that the optimal control is of bang-bang type consisting of a first phase of no action where no intervention measures are taken and a second phase where the maximum possible lockdown is applied until the end of the intervention time $T$. 
 In \cite{Balderrama22} the same problem is studied, broadening the previous results by adding an $L^1$ restriction on the control accounting for a restriction on the economic cost. As a consequence, in this case the optimal control can have three phases: a first phase of no action and a  second phase of partial lockdown during the maximum possible time that the $L^1$ restriction allows.
In the present article, we extend the results obtained in \cite{Balderrama22} by adding the peak prevalence constraint for all time.
 
Dealing with running state constraints in optimal control theory adds a layer of complexity to our approach. These constraints reflect the dynamic and evolving nature of healthcare capacity, making it essential to adapt control strategies in real-time to ensure the efficient allocation of resources while minimising the disease's spread. 

Several authors have analysed the problem of controlling an epidemic reducing transmission through lockdown policies, while keeping the infected below a certain threshold to prevent healthcare system saturation. Among them, Nenchev  \cite{Nenchev2020}  considered an SIRQ-type model with a Mayer-type cost function. On the other hand, Avram  \cite{Avram2022}, Sereno  \cite{Sereno2022}, and Miclo  \cite{Miclo2020} studied this issue using an SIR-type model and an $L^1$-type cost function for the control. Lastly, Morris \cite{Morris2021} analysed an SIR-type model with a peak prevalence-type cost function and  in \cite{Balderrama23} we studied the optimal final size with no $L^1$ restriction.
In all these articles the optimal control obtained is of boundary-bang type.

In this article we employ a cost function different from those used in the previously mentioned works. Specifically, we use a final size of epidemic cost function with the aforementioned restrictions.  Our goal is to characterise the optimal control for the problem of maximising the final size of the susceptibles with an $L^1$-type restriction on the control and a running state constraint for the infected. 
Using an optimal control approach 
we show that the optimal strategy is a single {\it bang-bang} or {\it boundary-bang} type with a boundary control that keeps the infected in its maximum possible value followed by a partial lockdown applied in a single interval of time.  Moreover, we characterise the time to start and finish each of these phases of the intervention.


In the following sections we shall elucidate the key parameters governing the behaviour of the considered SIR model, and present the optimal control strategies that can help mitigate the impact of epidemics under limited hospital capacity with real-time patient admission constraints and limited economic cost. Through rigorous analysis and simulation, we will show the potential of our approach to inform decision-making processes during public health crises.

The paper is organised as follows.  We start by describing the optimal control model and preliminaries results in section \ref{se: problem statement}.  In section \ref{se: pontryagin} we use Pontryagin's Maximum Principle  (PMP) and show that only bang-bang or boundary-bang controls are admitted in a particular sequence to solve the corresponding optimal control problem. To obtain the specific optimal policy, the switching times between the different phases are determined in an induced optimisation problem in section \ref{se: caracterization_optimal}. In section \ref{se: numerical} we perform a numerical test of the results by finding approximate numerical solutions of the control problem and comparing them with the rigorous solution summarised in Theorem~\ref{te: control_optimo_sigma1gral_rest_v2}.  Finally, in section \ref{se: summary} we give a summary of the main findings and discuss future work.


\section{Problem statement and preliminaries} 
\label{se: problem statement}

We address the challenge of managing the spread of an epidemic in the absence of vaccination and with limited hospital resources. The primary means of control available is isolation, represented by a time-dependent reproduction number, $\sigma(t)$, that can be adjusted within the range $[\sigma_s,\sigma_f]$. Here, $\sigma_s$ signifies stricter isolation (referred to as "strict" quarantine), while $\sigma_f$ indicates a state without restrictions ("free" from restrictions). The intervention is applied over a finite period $[0,T]$, where T represents the duration of the intervention. Following this period, restrictions are lifted, allowing the disease to spread freely with a constant $\sigma(t)=\sigma_{f}$ for all $t>T$. The parameter $\sigma(t)$ is conceptualised as encapsulating political measures like social distancing, business and institution lockdowns. Then, the system evolves according to the following set of coupled nonlinear ordinary differential equations:

\begin{subequations} \label{eq: SIR}
\begin{align}
\label{eq: SIR_x}
x'(t) &= -\gamma \, \sigma(t) \, x(t) \, y(t), \\
 \label{eq: SIR_iy}
 y'(t) &= \gamma \, \sigma(t) \, x(t) \, y(t) - \gamma \, y(t),
\end{align}
\end{subequations}
with $(x(0),y(0))\in {\cal{D}}=\left\{ (x_0,y_0): x_0>0, y_0>0, x_0+y_0\le 1\right\}$, $\sigma(t) \in [\sigma_s,\sigma_f]
$ for $t\in [0,T]$ and $\sigma(t)=\sigma_f$ for $t>T$, where $0\le \sigma_s<\sigma_f$.  Note that the region ${\cal {D}}$ is forward-invariant and there exists a unique solution for all time \cite{Hethcote2000}. 

We also assume that during the intervention period $[0,T]$, an extremely restrictive isolation for an extended duration is deemed impractical, leading to the restriction:
\begin{align} \label{eq: int_cond}
\int_0^T  \sigma(t) dt \ge \sigma_s \tau + \sigma_f(T-\tau),
\end{align}
with $\tau<T$ given, implying that for an optimal solution, the strict quarantine lasts at most $\tau$ time (see inequality \eqref{eq: cota_mu_F}).

To prevent an overload of treatment resources and considering that a fraction of infected individuals will need to be hospitalised, the constraint \(y(t) \leq K\) is imposed for \(t \in [0, T]\), where \(K > 0\) is a constant (running state constraint).

After the intervention at time \(T\), we calculate \(x_{\infty}(x(T), y(T), \sigma_f) = \lim_{t \to \infty} x(t)\), where \((x(t), y(t))\) is the solution of the system with initial condition \((x(T), y(T))\) and constant reproduction number \(\sigma(t) \equiv \sigma_f\) for \(t > T\).

Our objective is to determine the optimal control for the SIR model described above to minimise the overall impact of the pandemic in terms of the total number of infections. We aim to find an optimal \(\sigma(t)\)  in the family of piecewise continuous functions, that minimises the total damage, considering the long-term behaviour of the system.

As it is known,  for $y(0)>0$, $\gamma \ge 0$ and $\sigma \ge 0$ the fraction of infected $y(t)>0$ for any time $t \ge 
0$ and approaches zero asymptotically, i.e., $y_{\infty} \equiv \lim_{t \to \infty} 
y(t)= 0$. The fraction of susceptible individuals $x(t)$ is strictly decreasing, and $x_{\infty} \equiv \lim_{t \to \infty} x(t)\in  (0,1/\sigma)$ (see Theorem 2.1 in \cite{Hethcote2000}).  Therefore, the state of the system in the long time limit consists only 
of susceptible and recovered individuals, $x_{\infty} + z_{\infty} =1$, where $z_{\infty} \equiv \lim_{t\to \infty} z(t)$. 

In order to find the optimal $\sigma(t)$ it proves convenient to work with the 
fraction of susceptible individuals in the long term $x_{\infty}$ instead.  Then, given that minimising $z_{\infty}$ is 
equivalent to maximising $x_{\infty}$, we define the functional
\begin{align}
    \label{eq: Jcosto}
J(x,y,\sigma):=x_{\infty}(x(T),y(T),\sigma_f) 
\end{align}

Additionally, a new state variable \(v(t) = \int_{0}^{t} \sigma(s) \, ds\) is introduced, enabling the study of the optimal control model
\begin{subequations} \label{eq: SIR_iso}
\begin{align}
\label{eq: funcional_J_iso}
\max \quad& J(x,y,v,\sigma):=x_{\infty}(x(T),y(T),\sigma_f)  \\
\label{eq: SIR_iso_x}
s.t. \quad &x'(t)= -\gamma \sigma(t) x(t) y(t),   \quad x(0) = x_0, \quad t\in [0,T],\\
 \label{eq: SIR_iso_y}
 &y'(t) = \gamma \sigma(t) x(t) y(t) - \gamma y(t),  \quad y(0)= y_0, \quad t\in [0,T],\\
 \label{eq: SIR_iso_v}
 &v'(t)  =\sigma(t),  \quad v(0)= 0,  \quad t\in [0,T], \\
 \label{eq: sigma_S_T}
 &\sigma(t) \in [\sigma_s,\sigma_f], \quad \text{a.e. } t\in [0,T] \\
     \label{eq: rest_int}
  & v(T) \ge \sigma_s \tau + \sigma_f (T-\tau) \\
    \label{eq: restr_yt}
  &y(t)\le K , \quad t\in [0,T],
\end{align}
\end{subequations}

Next, we give some definitions for optimal control problems with running state constraints as outlined in  \cite{Maurer1977}.

An extremal arc of \eqref{eq: funcional_J_iso}-\eqref{eq: rest_int}  is called an  {\bf unconstrained extremal} whereas an extremal arc of  \eqref{eq: funcional_J_iso}-\eqref{eq: restr_yt} is called a {\bf constrained extremal}. For given $K$ the state constraint \eqref{eq: restr_yt} is called {\bf active} if the optimal unconstrained trajectory violates  \eqref{eq: restr_yt}.

By definition, the {\bf order $p$} of an inequality pure state constraint is the first derivative of the restriction, $h(x(t),y(t),v(t))=y(t)-K \le 0$, containing the control $\sigma$ explicitly.  
Since the control variable $\sigma$ appears in the first total time derivative of $h(x(t),y(t),v(t))$\begin{align} \label{eq: orden_rest}
\frac{dh(x(t),y(t),v(t))}{dt}&=y'(t)=\gamma y(t) (\sigma(t) x(t)-1),
\end{align}
the state constraint is of order $p=1$.

A subarc of the trajectory $(x(t),y(t),v(t))$ for which $h(x(t),y(t),v(t))<K$ is called an {\bf interior arc}; a {\bf boundary arc}  is a subarc of the trajectory where $h(x(t),y(t),v(t))=K$ for $a\le t\le b$ with $0<a<b<T$. Here $a$ and $b$ are called the {\bf entry-}  and {\bf exit-time} of the boundary arc; $a$ and $b$ are also termed {\bf junction points}. An arc is said to have a {\bf contact point} with the boundary at $c\in (0,T)$ if $h(x(c),y(c),v(c))=K$ and $h(x(t),y(t),v(t)) <K$ for all $t\neq c$ in a neighbourhood of $c$.

From \eqref{eq: orden_rest} we have that on boundary arcs (where $y(t)=K$ for $t\in [a,b]\subset [0,T]$), the boundary control is determined by 
\begin{align}
 \sigma_b(t)=\frac{1}{x(t)} \label{eq: sigmab}
\end{align}
for $t\in (a,b)$ (see Maurer \cite{Maurer1977}). From equation \eqref{eq: SIR_iso_x}  $x'(t)= - \gamma K$  and therefore
\begin{align} \label{eq: sigmabhiperb}\sigma_b(t)= \frac{1}{x(t_1)-\gamma K (t-t_1)} \text{  for   } t\in (a,b).\end{align}

In what follows we set forth Theorem 2 proved in \cite{Balderrama22} with $\kappa=0$ 
for the unconstrained optimal control problem given by equations \eqref{eq: funcional_J_iso}-\eqref{eq: rest_int}. 
In that article, it is proven that an optimal control is bang-bang and has the form
\begin{equation} \label{eq: sigma_viejo}
		\sigma[t,\mu](r):= \left\{ \begin{array}{ll}
			\sigma_f& \text{for }0\le r \le t,\\
			\sigma_s & \text{ for } t < r \le t+\mu,\\
			\sigma_f & \text{ for } t+\mu < r \le T,
		\end{array}\right.
	\end{equation}
with $t\ge 0, \mu \ge 0$ and $t+\mu \le T$. In addition, it is defined $(x[t,\mu],y[t,\mu])$ as the associated state variables, solution of the system \eqref{eq: SIR_iso_x}-\eqref{eq: SIR_iso_y} with $\sigma=\sigma[t,\mu]$, and the functions 
\begin{equation} \label{eq: omega}
w(t)= \left\{ \begin{array}{ll}  \int_{t}^{t+\tau } \frac{\sigma_f x[t,\tau](r)-1}{y[t,\tau](r)} dr  & \text{for }0\le t \le T-\tau \\
& \\
 \int_{t}^{T} \frac{\sigma_f x[t,T-t](r)-1}{y[t,T-t](r)}dr & \text{ for } T-\tau < t \le T,
\end{array}\right.
\end{equation} 
and 
\begin{align}\label{eq: alpha}
\alpha(t)&=
\frac{1}{\gamma y[t,T-t](t)}
 \end{align}
 for $t\in[T-\tau,T]$.

\begin{theorem}[Theorem 2 in \cite{Balderrama22} with $\kappa = 0$] \label{te: control_optimo_sigma1gral}
	Let $0\le \sigma_s<\sigma_f$ with $\sigma_s<1$ and $w$ be given by Eq.~\eqref{eq: omega}. Then the optimal control is unique and is given by
	\begin{equation} \label{eq: sigmacero_con_rest_int}
		\sigma^0(r):= \left\{ \begin{array}{ll}
			\sigma_f& \text{for }0\le r \le t^0,\\
			\sigma_s & \text{ for } t^0 < r \le t^0+\mu^0,\\
			\sigma_f & \text{ for } t^0+\mu^0 < r \le T,
		\end{array}\right.
	\end{equation}
	where
	\begin{itemize}
		\item[1.1.] For $w(0)\le 0$: $t^0=0$ and $\mu^0=\tau$.
		\item[1.2.] For  $w(0)>0$ and   $w(T-\tau)\le 0$: {$t^0=\hat{t}$} and $\mu^0=\tau$ where {$\hat{t}$} is the unique value on $[0,T-\tau]$ such that {$w(\hat{t})=0$}.
		\item[1.3.] For $0<w(T-\tau)\le\dfrac{1}{\gamma y[T-\tau,\tau](T-\tau)}$:  $t^0=T-\tau$ and $\mu^0=\tau$.
		\item[1.4.] For $w(T-\tau)>\dfrac{1}{\gamma y[T-\tau,\tau](T-\tau)}$: $t^0=\tilde{t}$ where $\tilde{t}$ is the unique value on $[T-\tau,T]$ such that $w(\tilde{t})=\dfrac{1}{\gamma y[\tilde{t},T-\tilde{t}](\tilde{t})}$ and $\mu^0=T-\tilde{t}$.
	\end{itemize}
\end{theorem}

We now compute the partial derivatives of $x_{\infty}(x,y,\sigma)$ with respect to
$x$ and $y$ in the same way that it is done in \cite{ketcheson2020optimal} and \cite{Balderrama22}. 
Recall that given $s>0$ we previously defined $x_{\infty}(x,y,\sigma)=\lim_{t \to \infty} x(t)\), where \((x(t), y(t))\) is the solution of the system with initial condition \((x, y)\) at time $s$ and constant reproduction number \(\sigma(t) \equiv \sigma\) for \(t > s\).

Finally, recall that the solution of the SIR model without control (that is with constant $\sigma(t)=\sigma$), satisfies that 
$$\rho(x(t),y(t),\sigma):= x(t) e^{-\sigma (x(t)+y(t))}$$
is constant in time and $y_{\infty}=0$ (see for instance \cite{Hethcote2000}). Then we have that $x_{\infty}=x e^{\sigma( x_{\infty}-x-y)}=\rho(x,y,\sigma)e^{\sigma x_{\infty}}$ and $w=-\sigma x_{\infty}$ satisfies the equation $we^{w}=-\sigma \rho(x,y,\sigma)$.  Therefore $w=W_0(-\sigma \rho(x,y,\sigma))$ where $W_0$ is the principal branch of Lambert's $W-$function \cite{Pakes2015}, and thus for any $(x,y)\in {\cal D}$
$$x_{\infty}(x,y,\sigma)=-\frac{1}{\sigma} W_0(\sigma \rho(x,y,\sigma)).$$

From this expression we can compute the partial derivatives of $x_{\infty}(x,y,\sigma)$ with respect to
$x$ and $y$ (see \cite{ketcheson2020optimal}):
\begin{subequations} \label{eq: deriv_xinfty}
\begin{align} \label{eq:deriv_xinfty_x}
\frac{\partial x_{\infty}(x,y,\sigma)}{\partial x}&=\frac{1-\sigma x}{x} \frac{x_{\infty}(x,y,\sigma)}{1-\sigma x_{\infty}(x,y,\sigma)}, \\
 \label{eq:deriv_xinfty_y}
\frac{\partial x_{\infty}(x,y,\sigma)}{\partial y}&=-\frac{\sigma x_{\infty}(x,y,\sigma)}{1-\sigma x_{\infty}(x,y,\sigma)}.
\end{align}
\end{subequations}

\section{Necessary conditions for an optimal control}
\label{se: pontryagin}
To derive first order necessary optimality conditions, we apply  
 Pontryagin's Maximum Principle for an optimal control problem with a pure state constraint.
 We assume that the state constraint is not active at the initial and final time, i.e., $y(0) < K, y(T) < K$. The corresponding augmented Hamiltonian $H$ is given by
\begin{align} \label{Hamiltonian}
H(x,y,\sigma,\lambda)= -\lambda_1 (\gamma  \sigma x y) +\lambda_2 (\gamma \sigma x y -\gamma y )+\lambda_3 \sigma
\end{align}
and the Lagrangian $L$
\begin{align} \label{Lagrangian}
L(x,y,\sigma,\lambda,\eta)= H(x,y,\sigma,\lambda)+\eta (y-K),\end{align}

where  $\lambda \in \mathbb{R}^2$ and $\eta\ge 0$. We have omitted the state $v$ since $H$ is independent of this variable.
Given an optimal solution
$(x^*,y^*, v^*,\sigma^*)$, the
necessary conditions for a maximum process $(x^*,y^*,v^*,\sigma^*)$ on $[0,T]$ 
are the following (see \cite{Clarke2013,Hartl1995}):

\begin{quote}
There exists a real number $\lambda_0 \ge 0$, the adjoint variable $\lambda:[0,T] \to \mathbb{R}^3$, which is absolutely continuous, such that 
\begin{align} \label{eq: multiplicadores_cero}
(\lambda_0,\beta,\lambda(t),\eta(t))\neq 0, \quad \text{ for all } t \in [0,T]
\end{align}
 and the following conditions hold:
\begin{enumerate}
\item The adjoint variables $\lambda_i(t)$ satisfy  a.e. $t\in [0,T]$
\begin{subequations} \label{eq: pmp}
\begin{align}
\label{eq: pmp1}
\lambda_1'(t)& = (\lambda_1(t) -\lambda_2(t)) \gamma \sigma^*(t) y^*(t),\\
\label{eq: pmp2} 
\lambda_2'(t)& = (\lambda_1(t)  -\lambda_2(t)) \gamma \sigma^*(t)x^*(t)+\gamma \lambda_2(t)+\eta(t) \\
\label{eq: lambda_expl3}
\lambda_3'(t)& =0,  
\end{align}
with final time conditions (using equation \eqref{eq: deriv_xinfty} and the abbreviation $x_{\infty}$ for $x_{\infty}(x(T),y(T),\sigma_f)$)
\begin{align}
\lambda_1(T)&=\lambda_0 \frac{\partial x_{\infty}}{\partial x(T)}=\lambda_0 \frac{1-\sigma_f x^*(T)}{x^*(T)} \frac{x_{\infty}}{1-\sigma_f x_{\infty}}, \label{eq: pmp3} \\
 \lambda_2(T)&=\lambda_0 \frac{\partial x_{\infty}}{\partial y(T)}=-\lambda_0\frac{\sigma_f x_{\infty}}{1-\sigma_f x_{\infty}}. \label{eq: pmp4} \\
 \lambda_3(T)&=\beta \ge 0 \text{  and  } \lambda_3(T) (v^*(T)-\sigma_f (T-\tau) - \sigma_s \tau)=0. \label{eq: lambda_final3}
 \end{align}

\item The function $\eta \ge 0$ satisfies 
\begin{align} \label{eq: pmp5}
\eta(t) (y^*(t)-K)=0 ,
\end{align}
and is continuous on the interior of each boundary arc.
\item  The jump condition at a contact point 
or junction point $t$ 
is
 \begin{align} \label{eq: pmp6_a}
 \lambda_1(t^+)&=\lambda_1(t^-), \\
\lambda_2(t^+)&=\lambda_2(t^-)+\nu, \,\, \nu \ge 0,   \label{eq: pmp6_b} \\
H(t^-)&=H(t^+). \label{eq: pmp6_c}
\end{align}
\item  For a.e. $t\in [0,T]$
\begin{align*}
H(x^*(t),y^*(t),\sigma^*(t),\lambda(t)) = \max_{\sigma: \sigma_s\leq \sigma \le \sigma_f} H(x^*(t),y^*(t),\sigma,\lambda(t)). 
\end{align*}
We can rewrite
$$H(x^*,y^*,\sigma^*,\lambda)=\phi \sigma^* -\gamma\lambda_2  y $$
where
\begin{align} \label{eq: phi}
\phi(t)=\beta+\gamma x^*(t) y^*(t) (\lambda_2(t)-\lambda_1(t)),
\end{align}
obtaining
\begin{align} \label{eq: pmp7}
& \sigma^*(t) \phi(t) =\max_{\sigma: \sigma_s\leq \sigma \le \sigma_f} \sigma  \phi(t). \end{align}
\item There exists a constant $C$ such that for a.e. $t\in [0,T]$
\begin{align*} 
H(x^*(t),y^*(t),\sigma^*(t),\lambda(t)) =C,
\end{align*}
that is
\begin{align} \label{eq: pmp8}
 \phi(t) \sigma^*(t) -\gamma \lambda_2(t) y^*(t) =C.
\end{align}
\end{subequations}
\end{enumerate}

\end{quote}

As the Hamiltonian is linear in $\sigma(t)$, the optimal control will 
depend on the sign of the switching $\phi$ function, given by \eqref{eq: phi} (see lemma \ref{le: optimo es bb sing}). 
If $\phi(t)=0$ on a subinterval $I \subset [0,T]$ the control $\sigma$ is called {\bf singular} on $I$. 

\begin{lemma} \label{le: optimo es bb sing}
Let $\sigma^*$ be an
optimal control, $(x^*,y^*)$ the associated state given by \eqref{eq: SIR_iso_x} and \eqref{eq: SIR_iso_y} with $
\sigma=\sigma^*$ and $(\lambda_1,\lambda_2)$ solution of equations 
\eqref{eq: pmp1}-\eqref{eq: pmp2} and \eqref{eq: pmp3}-\eqref{eq: pmp4}.  Then
\begin{equation} \label{eq: sigma_bb1}
\sigma^*(t)= \left\{ \begin{array}{ll}
\sigma_f & \text{ if } \phi(t)>0\\
\sigma_{sing}(t) & \text{ if } \phi(t)=0\\
\sigma_s& \text{ if } \phi(t)<0.
\end{array}\right.
\end{equation}
\end{lemma}

\begin{proof}
From the optimality condition \eqref{eq: pmp7} we have that for a nonsingular arc, that is $\phi(t)\neq 0$, 
\begin{equation} \label{eq: sigma_bb}
\sigma^*(t)= \left\{ \begin{array}{ll}
\sigma_f & \text{ if } \phi(t)>0\\
\sigma_s& \text{ if } \phi(t)<0.
\end{array}\right.
\end{equation}
Calling $\sigma_{sing}$ the optimal control corresponding to a singular control 
we obtain \eqref{eq: sigma_bb1}.
\end{proof}

In order to characterise the behaviour of the switching function $\phi$  we analyse 
some properties of the interior and boundary arcs.
Using \eqref{eq: SIR_iso_x}-\eqref{eq: SIR_iso_y} and \eqref{eq: pmp1}-\eqref{eq: pmp2}  we compute the derivative of $\phi$ for a.e. $t$:
\begin{align} \label{eq: phiprima}
\phi'(t)=\gamma x^*(t) y^*(t) (\gamma \lambda_1(t) + \eta(t)).
\end{align}

 {\bf Interior arcs}. Assume $(t_1,t_2)\subset [0,T]$ is an interior arc, that is $y^*(t)<K$ for all $t\in (t_1,t_2).$
Then, from \eqref{eq: pmp5}  we have that $\eta\equiv 0$ in $(t_1,t_2)$. 

We now study the existence of singular controls for interior arcs.
Assume $\phi(t)=0$ on an interval $[a,b]\subset (t_1,t_2)$, then using that $\eta \equiv 0$, from \eqref{eq: phiprima}
 we obtain
\begin{align}
\label{eq: singular3}
0&=\phi'(t)=\gamma^2 \lambda_1(t)  x^*(t) y^*(t) \text{ a.e. } t\in (a,b),
\end{align}
from where we have that $\lambda_1\equiv 0$ on $(a,b)$. Thus, from \eqref{eq: pmp1} $\lambda_2\equiv  0$ on $(a,b)$, concluding that
$\lambda_1\equiv \lambda_2 \equiv 0$ on $(t_1,t_2)$.
As a consequence we have the following remark.
\begin{remark}\label{re: no hay singular arc}
Given $(t_1,t_2)$ an interior arc, if there exists $t\in (t_1,t_2)$ such that $(\lambda_1(t),\lambda_2(t))\neq (0,0)$, 
there cannot be singular arcs in $(t_1,t_2)$.
\end{remark}

 {\bf Boundary arcs}. Assume $y^*(t)=K$ for all $t\in [t_1,t_2]$ and that the 
control corresponding to the
boundary arc computed in \eqref{eq: sigmab} satisfies $\sigma_b(t)\in (\sigma_s,\sigma_f)$ for all $t\in (t_1,t_2)$.
Since $\sigma_s<1$ and $x(t)<1$, from \eqref{eq: sigmab} the latter condition can be reduced to $\sigma_b(t)<\sigma_f$  for all $t\in (t_1,t_2)$. 
Then, from the optimality condition \eqref{eq: pmp7} we deduce that $\phi(t)\equiv 0$ on $(t_1,t_2)$ 
concluding that a boundary arc is a singular arc as shown in \cite{Maurer1977} for controls appearing linearly.
From \eqref{eq: phiprima} we derive that $\eta(t)=-\gamma \lambda_1(t)$ on $(t_1,t_2)$ 
and from \eqref{eq: sigmab} and \eqref{eq: pmp2} we obtain 
$\lambda_2'(t) =0 $ on $(t_1,t_2)$.

\begin{lemma} \label{le: triangulo}
If $(t_1,t_2)$ is a boundary arc, then 
$(t_1,t_2)\subset \{t\in [0,T]:  \lambda_2(t)\le \lambda_1(t)\le 0\}$. Moreover we have that $(\lambda_1(t),\lambda_2(t))\neq (0,0)$ for all $t\in(t_1,t_2)$ or $\lambda_1\equiv \lambda_2\equiv 0$ on $(t_1,t_2)$. \end{lemma}
\begin{proof}
Let  $(t_1,t_2)$ be a boundary arc.
Since $\eta\ge 0$, from equality  $\eta=-\gamma \lambda_1$ on boundary arcs we obtain $\lambda_1(t) \le 0$ for all $t\in (t_1,t_2)$. 
In addition, for $t$ such that $\lambda_1(t)<\lambda_2(t)$ it holds that $\phi(t)>0$ implying that there cannot be singular arcs 
on $\{t\in [0,T]:  \lambda_1(t)<\lambda_2(t)\}$ and thus neither 
boundary arcs. 
Moreover, from \eqref{eq: pmp8} and using that $\phi(t)=0$  for all $t\in (t_1,t_2)$, we have that $\lambda_2\equiv -C/(\gamma K)$ is constant. Thus, if $(\lambda_1,\lambda_2)(s)=(0,0)$ for some $s\in (t_1,t_2)$, we would have that  $\lambda_2\equiv 0$ on $(t_1,t_2)$ and from equation \eqref{eq: pmp1}, $\lambda_1$ is a solution of the linear equation
$
\lambda_1'(t) = \gamma K  \sigma_b(t)  \lambda_1(t)
$ 
on $(t_1,t_2)$ 
with $\lambda_1(s)=0$ concluding that $\lambda_1\equiv 0$, yielding also $\eta\equiv 0$.
 \end{proof}

From the aforementioned we have that
\begin{equation} \label{eq: phiprima2}
\phi'(t)= \left\{ \begin{array}{ll}
\gamma^2 \lambda_1(t)  x^*(t) y^*(t) & \text{ for nonsingular (interior) arcs } \\
0 \, \, (\text{and} \, \, \phi\equiv 0)& \text{ for singular (interior or boundary) arcs }.
\end{array}\right.
\end{equation}

\begin{remark}\label{eq: ymenork} Note that $y^0(t^0)$ is the global maximum of $y^0$ on $[0,T]$ and $t^0\le t_m$ where $x^0(t_m)=1/\sigma_f$. In fact, from \eqref{eq: SIR_iso_y} we have that $y(t)$ is decreasing for $t$ such that $x(t)< 1/\sigma_f$, proving that $t^0 \le t_m$.
Moreover, for the cases  {\it 3.} and {\it 4.} in theorem \ref{te: control_optimo_sigma1gral},  $t^0+\mu^0=T$ and therefore the optimal control $\sigma(t)=\sigma_f$ for $t\in [0,t^0)$ and $\sigma(t)=\sigma_s$ for  $t\in (t^0,T]$. 
On the other hand, for the cases {\it 1.} and {\it 2.} observe that  $w(t^0)\le 0$, and thus, from \eqref{eq: omega}  and \eqref{eq: SIR_iso_x}, it must hold  $x^0(t^0+\mu^0)\le1/\sigma_f$, from where we derive the desired result.
\end{remark}
Given $y^0$ as in Theorem \ref{te: control_optimo_sigma1gral}, we consider two cases. 

{\bf Case 1}: $y^0(t)\le K$ for all $t\in [0,T]$. 
This mean that the optimal solution of problem \eqref{eq: funcional_J_iso}-\eqref{eq: rest_int} without 
the running state constraints satisfies also this latter condition \eqref{eq: restr_yt} implying that is also 
an optimal solution of the complete problem \eqref{eq: SIR_iso}.
In this case we will have that there is no boundary arc.

{\bf Case 2}: {$\displaystyle \max_{[0,T]} y^0(t)>K$}, that is, exists $t\in [0,T]$ such that $y^0(t)>K$. 
From Remark \ref{eq: ymenork} $y^0$ attains its maximum at $t^0$. Then, we can define 
 \begin{equation} \label{eq: tb}
 t_b= \min \left\{ t\in [0,t^0]: y^0(t)=K\right\}
 \end{equation} 
 satisfying $t_{b}\in (0,t^0)$ where we have used that the initial condition $y_0<K$.


It is worth noticing that since $\sigma(t)\in [\sigma_s,\sigma_f]$ for all $t\in [0,T]$, from \eqref{eq: SIR_iso_y} it holds $y'(t)<0$ for all $t>r$ where $x(r)=1/\sigma_f$, from where we derive that there cannot be any junction or contact point $s>r$.

\begin{lemma} \label{le: phi es continua}
Assuming that for any contact or exit point $s$, $x^{*}(s)>\frac{1}{\sigma_f}$, then the switching function $\phi$ and $\lambda_2$ are continuous.
\end{lemma}

\begin{proof}
From \eqref{eq: pmp6_b} we have that for any junction or contact point $s$ it holds:
\begin{align} \label{eq: saltolambda2}
\lambda_2(s^+)-\lambda_2(s^-)=\nu
\end{align}
with $\nu \ge 0$ implying
\begin{align} \label{eq: saltophi}
\phi(s^+)-\phi(s^-)=\gamma x^*(s) y^*(s) \nu \ge 0.
\end{align}
When $s$ is a contact point such that $x^*(s)> \frac{1}{\sigma_f}$, from \eqref{eq: sigma_bb}, it holds $\phi(s^-)\ge 0$
and $\phi(s^+)\le 0$ respectively. Thus, from \eqref{eq: saltophi} we obtain $\phi(s^+)=\phi(s^-)=0$.
If $s$ is an entry point, then $\phi(s^+)=0$ implying from \eqref{eq: saltophi} that $\phi(s^-)\le 0$. Assume $\phi(s^-)<0$, then from \eqref{eq: sigma_bb}, $\sigma(s^-)=\sigma_s$ and therefore $y'(s^-)<0$  contradicting that $s$ is an entry point. It follows that $\phi(s^+)=\phi(s^-)=0$.
Finally, if $s$ is an exit point, then $\phi(s^-)=0$ and  from \eqref{eq: saltophi} $\phi(s^+)\ge 0$. Assume $\phi(s^+)>0$, then from \eqref{eq: sigma_bb} $\sigma(s^+)=\sigma_f$ and therefore, using that $x^*(s)> \frac{1}{\sigma_f}$, we have  $y'(s^+)>0$ yielding a contradiction.

It follows that $\phi$ is continuous and therefore from  \eqref{eq: saltolambda2} $\lambda_2$ is also a continuous function.
\end{proof}

Moreover, we have the following lemma.
\begin{lemma} \label{le: extremal restringido tiene arco frontera}
If {$\displaystyle \max_{[0,T]} y^0(t)>K$} and any contact point $s$ of the constraint extremal satisfies $x^*(s)>\frac{1}{\sigma_f}$, then 
the constrained extremal has at least one boundary arc.
\end{lemma}

\begin{proof}
If the constraint extremal satisfies $y^*(t)<K$ for all $t\in [0,T]$, then there are not any junction points and therefore the first order necessary conditions \eqref{eq: pmp1}-\eqref{eq: pmp8} hold for $\eta \equiv 0$ and $\lambda_2$ a continuous function. 
On the other hand, if the constraint extremal has only contact points, $\eta\equiv 0$ almost everywhere and from Lemma \ref{le: phi es continua}, $\lambda_2$ is continuous. 
Thus, in both cases, the constraint extremal satisfies also the necessary conditions of the maximum principle for the unconstrained extremal, contradicting its uniqueness. This proves the desired result.
\end{proof}

We now prove that the problem is normal. 

\begin{lemma}\label{le: es_normal}
If the constraint extremal satisfies $y^*(T)<K$ and any contact point $s$ satisfies $x^*(s)>\frac{1}{\sigma_f}$, then 
the optimal control problem \eqref{eq: SIR_iso} is normal, that is $\lambda_0>0$.
\end{lemma}

\begin{proof}
Assume $\lambda_0=0$, then from \eqref{eq: pmp3} -\eqref{eq: pmp4}, $\lambda_1(T)=\lambda_2(T)=0$. Since $y^*(T)<K$, there exists $\delta>0$ such that $y^*(t)<K$ for all $t\in(T-\delta,T]$ and therefore, from \eqref{eq: pmp5}, $\eta(t)=0$ for all $t\in (T-\delta,T]$. Thus, from \eqref{eq: pmp1}-\eqref{eq: pmp2} we have $\lambda_1(t)=\lambda_2(t)=0$ for all $t\in (T-\delta,T]$ and therefore, from \eqref{eq: phi}, $\phi(t)=\beta\ge 0$  for all $t\in (T-\delta,T]$. Using the non-triviality condition \eqref{eq: multiplicadores_cero}, we obtain $\beta>0$ and from Lemma \ref{le: phi es continua}, $\phi$ is a continuous function yielding there cannot be a boundary arc on any interval of the form $[r,T-\delta]$ since in this case we would have $\phi(t)=0$ for all $t\in (r,T-\delta)$. 
Thus we conclude that $y^*(t)<K$ for almost every $t\in [0,T]$ implying $\phi(t)=\beta>0$ for all $t\in [0,T]$. Then, from \eqref{eq: sigma_bb} we have $\sigma^*(t)=\sigma_f$ for all $t\in [0,T]$ contradicting the complementarity condition \eqref{eq: lambda_final3}. 
This proves the statement, then we can consider $\lambda_0=1$.

\end{proof}
In the following lemma we prove the general expression of an optimal control.

\begin{lemma} \label{le: dos saltos}
If {$\displaystyle \max_{[0,T]} y^0(t)>K$} and any contact or exit point $s$ of the constraint extremal satisfies $x^*(s)>\frac{1}{\sigma_f}$, then the generalised optimal solution $\sigma$ of problem \eqref{eq: SIR_iso} is given by
\begin{equation} \label{eq: sigma opt gral}
\sigma(t)= \left\{ \begin{array}{ll}
\sigma_f & \text{ for  } 0\le t \le t_1 \\
\sigma_b(t)=\frac{1}{x(t)}& \text{ for  } t_1< t \le t_2 \\
\sigma_s & \text{ for  } t_2< t \le t_3 \\
\sigma_f&  \text{ for  } t_3< t \le T \\
\end{array}\right.
\end{equation}
where $0\le t_1 < t_2 \le t_3 < T$, $t_1,t_2 \in \{t\in [0,T]:  \lambda_2(t)\le \lambda_1(t)\le 0\}$ and $y(t)=K$ for all $t\in[t_1,t_2]$. 
\end{lemma}

\begin{proof}
From Lemmas \ref{le: phi es continua} and \ref{le: extremal restringido tiene arco frontera} we have that the constraint extremal has at least one boundary arc and $\lambda_2$ and $\phi$ are continuous functions.

The proof follows by analysing the phase diagram of $\lambda_1,\lambda_2$.
We begin by noting that a solution $(\lambda_1,\lambda_2)$ of the system 
\eqref{eq: pmp1}-\eqref{eq: pmp2} with final conditions \eqref{eq: pmp3}-\eqref{eq: pmp4} cannot cross both semilines 
$\lambda_1=\lambda_2\ge 0$ and $\lambda_1=\lambda_2<0$ or $\lambda_1=\lambda_2> 0$ and $\lambda_1=\lambda_2\le 0$. We prove the first case, the other is completely analogous.
In fact, assume there exist $s_1,s_2 \in [0,T]$ such that $\lambda_1(s_1)=\lambda_2(s_1)\ge 0$ 
and $\lambda_1(s_2)=\lambda_2(s_2)<0$.  
Evaluating \eqref{eq: phi} on $t=s_i$ for $i=1,2$ we have that $\phi(s_i)=\beta \ge 0$ and 
from \eqref{eq: pmp8}, $\lambda_2(s_i)=\frac{\beta \sigma^*(s_i)-C}{\gamma y^*(s_i)}$. 
 If $\beta=0$ it holds that $\lambda_2(s_i)=-\frac{C}{\gamma y^*(s_i)}$ and
if  $\beta>0$, using \eqref{eq: sigma_bb1}, we have that
 $\lambda_2(s_i)=\frac{\beta \sigma_f -C}{\gamma y^*(s_i)}$,  both cases contradicting that 
 $\lambda_2(s_1)$ and $\lambda_2(s_2)$ had opposite signs.

Since we have end time conditions on $T$ we go backwards from $(\lambda_1(T),
\lambda_2(T))$ with $\lambda_1(T)>0$ and $\lambda_2(T)<0$ (see \eqref{eq: pmp3} and \eqref{eq: pmp4}).
 Then, the solution $(\lambda_1,\lambda_2)$ stays in the set $\left\{t\in[0,T]: \lambda_1(t)>0, \lambda_2(t)<0\right\}$ either for all $t\in [0,T]$ or for some maximal interval to the left of $T$, named $(T-\delta,T]$.
For the latter case we have $\lambda_1(T-\delta)\le 0$ or $ \lambda_2(T-\delta)\ge 0$ and from Lemma \ref{le: triangulo}, there are only interior arcs on $(T-\delta,T]$, concluding 
from \eqref{eq: pmp5} that $\eta \equiv 0$ on $(T-\delta,T]$. Moreover, from Remark \ref{re: no hay singular arc}, we conclude that there cannot be singular arcs on $(T-\delta,T]$. 

From \eqref{eq: pmp1} and \eqref{eq: pmp2}, for $\lambda_2=0$ and $\lambda_1>0$, both $\lambda_1$ and $\lambda_2$ are increasing, showing that the trajectory backwards in time cannot go from the fourth to the first quadrant.

Since  $\lambda_1$ is increasing for $\lambda_1>\lambda_2$, the trajectory moves backwards in time to the left in the direction of the third quadrant. Moreover, it cannot touch the origin. In fact, 
assume $\lambda_1(T-\delta)=\lambda_2(T-\delta)=0$, using that $\eta(t)=0$ for $t\in (T-\delta, T]$ from \eqref{eq: pmp1}-\eqref{eq: pmp2}, we would have that $\lambda_1(t)=\lambda_2(t)=0$ for all $t\in [T-\delta, T]$  contradicting the end time conditions \eqref{eq: pmp3}-\eqref{eq: pmp4}.

In case the trajectory crosses the semiline $\lambda_1=0, \lambda_2<0$, we call $\tilde{t}\in [0,T)$ such that $\lambda_1(\tilde{t})=0$. From  \eqref{eq: phiprima2} we would have that $\phi$ has a local minimum on $\tilde{t}$.

Also, on the region $\lambda_2\le \lambda_1 \le 0$, from \eqref{eq: phiprima2}, $\phi$ is zero or a 
strictly decreasing function and from lemma \ref{le: extremal restringido tiene arco frontera} there exists a boundary arc in this region.

For the aforementioned and \eqref{eq: phiprima2} the generalised continuous function $\phi$ 
has the following structure (see Figure \ref{fig: phi_y_sigma})

\begin{equation} \label{eq: phi_tramos}
\phi(t)= \left\{ \begin{array}{ll}
\text{positive and decreasing} & \text{ for  } 0\le t \le t_1 \\
0& \text{ for  } t_1\le t \le t_2 \\
\text{negative and decreasing} & \text{ for  } t_2\le t \le \tilde{t} \\
\text{negative and increasing}&  \text{ for  } \tilde{t} \le t \le t_3 \\
\text{positive and increasing}&  \text{ for  } t_{3}\le t \le T \\
\end{array}\right.
\end{equation}
where $0\le t_1 < t_2 \le \tilde{t} \le t_3 \le T$ and $t_1,t_2 \in \{t\in [0,T]:  \lambda_2(t)\le \lambda_1(t)\le 0\}$. Note that if  $\phi(T)<0$ (for instance when $ \beta =0$) then $t_3=T$.
The conclusion follows  from Lemma \ref{le: optimo es bb sing}. 

\begin{figure}[H]
\begin{center}
\includegraphics{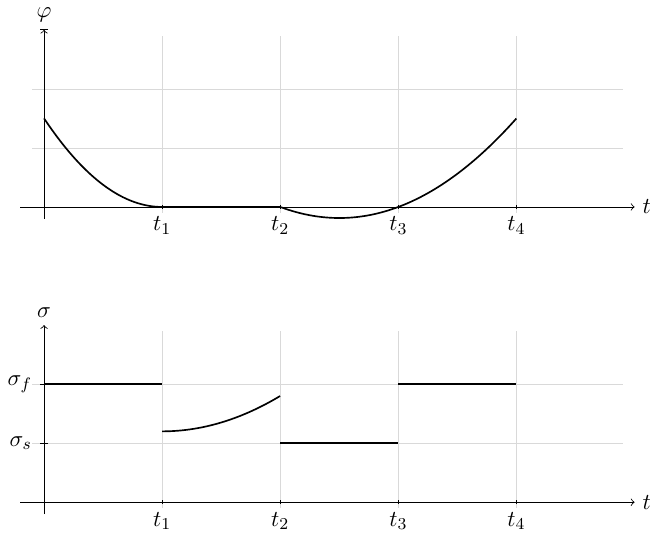}
\caption{Function $\phi$ and the corresponding $\sigma$ given by \eqref{eq: sigma opt gral}.}
\label{fig: phi_y_sigma}
\end{center}
\end{figure}
\end{proof}


\section{Characterisation of the optimal control}
\label{se: caracterization_optimal}

In this section we will give the main result (Theorem \ref{te: control_optimo_sigma1gral_rest_v2}) of this article.  We recall  that under the hypothesis of Lemma \ref{le: dos saltos} and using the notation $t_3= t_2 +\mu$ with $\mu\in[0,T-t_2]$, 
 if $\sigma$ is an optimal control, it has the form 
\begin{equation} \label{eq: sigma opt gral_kappa0}
\sigma[t_1,t_2,\mu](t)= \left\{ \begin{array}{ll}
\sigma_{f} & \text{ for  } 0\le t \le t_1 \\
\sigma_b(t)=\frac{1}{x(t)}& \text{ for  } t_1 < t \le t_2 \\
\sigma_{s} & \text{ for  } t_2 < t \le  t_2 +\mu \\
\sigma_{f} & \text{ for  } t_2+\mu < t \le T 
\end{array}\right. .
\end{equation}
In order to determine the optimal solution we need to establish the times $t_1^*,t_2^*,\mu^*$ that makes the control optimal.

For $t_1,t_2,\mu$ let $x[t_1,t_2,\mu]$ be the associated state, solution of equation 
\eqref{eq: SIR_iso_x} with $\sigma=\sigma[t_1,t_2,\mu]$. 
With the aim of analysing for which $t_1,t_2,\mu$  the integral restriction \eqref{eq: rest_int} is satisfied, 
we compute the integral of $\sigma[t_1,t_2,\mu]$ on $[0,T]$ and clearing the variable $\mu$ 
we obtain a linked between the parameters $t_1$, $t_2$ and $\mu$ given by
\begin{equation} \label{eq: mu_t1_t2}
\mu \le F[t_1](t_2).
\end{equation}
where $F[t_1]:[t_1,T] \to \mathbb{R}$ is a function of $t_2$ defined by:
\begin{equation} \label{eq: Ft1t2}
F[t_1](t_2):=\tau +  \frac{\sigma_f ( t_2 - t_1)}{\sigma_s-\sigma_f}
+  \frac{1}{\gamma K(\sigma_s-\sigma_f )} \ln \Big({1 - \frac{\gamma K (t_2-t_1)}{ x(t_1)} }\Big).
\end{equation}
Note that since $x[t_1,t_2,\mu](t_2)=x(t_1)-\gamma K (t_2-t_1)$ we have that the argument of the logarithm is positive showing that $F[t_1]$ is well defined.

Thus, the admissible controls $\sigma[t_1,t_2,\mu]$ given by \eqref{eq: sigma opt gral_kappa0} are characterised by the condition $(t_1,t_2,\mu)\in {\mathcal R}$ where
\begin{equation}
 \label{def: R}
{\mathcal R}=\left\{ (t_1,t_2,\mu)\in \mathbb{R}^3: 0\le t_1\le t_2, \,  0\le \mu \le F[t_1](t_2), \, \mu+t_2\le T)\right\}.
\end{equation}

In what follows, we will analyse {\bf Case 2} (see the previous section), that is, we will assume there exists $t\in [0,T]$ such that $y^0(t)>K$. In this case we will begin by proving that the time $t_1^*$ that makes the control $\sigma[t_1^*,t_2,\mu]$ optimal for some $t_2, \mu$ is given by $t_1^*=t_b$ (see Lemma \ref{le: t1*_es_tb}). Latter we will study the region ${\mathcal R}$ (see equation \eqref{eq: regionR}) and restrict the problem of finding $(t_2^*,\mu^*)$ for which $\sigma[t_1^*,t_2^*,\mu^*]$ is optimal.

\begin{lemma} \label{le: t1*_es_tb}
If {$\displaystyle \max_{[0,T]} y^0(t)>K$} and any contact or exit point $s$ of the constraint extremal satisfies $x^*(s)> \frac{1}{\sigma_f}$, then the time $t_1^*=t_b$ with $t_b$ defined in \eqref{eq: tb} makes the control $\sigma[t_1^*,t_2,\mu]$ optimal for some $t_2, \mu$ such that $(t_1^*,t_2,\mu) \in {\mathcal R}$.
\end{lemma}

\begin{proof}
From \eqref{eq: sigma opt gral_kappa0} and the restriction  \eqref{eq: restr_yt}, it is clear that $t_1^* \le  t_{b}$. If the latter inequality is strict, there cannot be a boundary arc contradicting Lemma  \ref{le: extremal restringido tiene arco frontera}.\end{proof}

From now on, we will assume that $t_1^*$ is fixed and given by $t_1^*=t_{b}$. It remains to characterise $t_2^*$ and $\mu^*$. In order to do that, we begin by restricting possible $t_2>t_1^*$ for which $\sigma[t_1^*,t_2,\mu]$ is admissible. 

\begin{remark} \label{re: x2_y_sigmaf}
 In order for $\sigma[t_1^*,t_2,\mu]$ to be an admissible control, the boundary control must satisfy $\sigma_b(r)=1/x[t_1^*,t_2,\mu](r) \in (\sigma_s,\sigma_f)$ for all $r\in (t_1^*,t_2)$. 
Then, since $x[t_1^*,t_2,\mu]$ is decreasing the condition $x[t_1^*,t_2,\mu](t_2)\ge 1/\sigma_f$ have to be fulfilled and in particular  $x[t_1^*,t_2,\mu](t)> 1/\sigma_f$ for all $t<t_2$. 
Using that for $t\in (t_1^*,t_2)$,  $y[t_1^*,t_2,\mu](t)=K$ and $x[t_1^*,t_2,\mu](t)=-\gamma K (t-t_1^*)+x[t_1^*,t_2,\mu](t_1^*)$, we obtain that the condition $x[t_1^*,t_2,\mu](t_2)\ge 1/\sigma_f$ is satisfied if and only if 
$t_2\le t_1^*+ \frac{x[t_1^*,t_2,\mu](t_1^*)-1/\sigma_f}{K\gamma}$.
\end{remark}

Then we compute, if there is any, the time $t$ where the trajectory given by equations \eqref{eq: SIR_iso_x}-\eqref{eq: SIR_iso_y} with $\sigma=\sigma[t_1^*,T,0]$
satisfies that $x[t_1^*,T,0](t)=1/\sigma_f$ and define
\begin{align}
 \label{eq: tm}
 t_{m}
 & = \max\left\{t\in [t_1^*,T]: x[t_1^*,T,0](t) \geq 1/\sigma_f \right\}
\end{align}
Note that if $t_m<T$, then $t_m=t_{1}^*+ \frac{x[t_1^*,T,0](t_1^*)-1/\sigma_f}{K\gamma}$. Also from \eqref{def: R}, since $0 \le \mu \le F[t_1^*](t_2)$, we define, if there is any, the zero of $F[t_1^*]$ given by  \eqref{eq: Ft1t2}
\begin{equation}
 \label{eq: tf}
 t_f=\max\left\{t\in [t_1^*,t_{m}]: F[t_{1}^*](t) \geq 0 \right\}
\end{equation}
and so we consider the restriction $F[t_{1}^*]:[t_{1}^*,t_{f}] \to \mathbb{R}$.
Moreover, in what follows we analyse how conditions involving $\mu$ and $t_2$ on ${\mathcal R}$ in \eqref{def: R} are related.
Using that $F[t_1^*](t_2)+t_2$ is increasing as a function of $t_2$ on $[t_{1}^{*},t_{f}]$ (see Lemma \ref{le: H_crece} in Supplement) we can define 
\begin{align} \label{def: tc}
t_c\in[t_1^*,t_{f}] \text{ such that } &F[t_{1}^{*}](t_{2}) < T-t_2 \,\forall t_2\in [t_1^*,t_{c}) \nonumber \\
 \text{ and } &F[t_{1}^{*}](t_{2}) > T-t_2 \, \forall t_2\in (t_{c},t_{f}].
 \end{align}
 
Note that if $F[t_{1}^{*}](t_{2}) > T-t_2$ for all $t_2\in (t_1^*,t_{f}]$, then $t_{c}=t_{1}^*$ and if $F[t_{1}^{*}](t_{2}) <T-t_2$ for all $t_2\in [t_1^*,t_{f})$, then $t_{c}=t_{f}$.

From the definition of $t_{c}$ we can define the continuous function $G:[t_{1}^{*},t_{f}] \to \R$
\begin{equation} \label{eq: Gpartida}
G(t_2)= \min\{ F[t_{1}^{*}](t_2), T-t_2\},
\end{equation}
which satisfies $G(t_{c})=F[t_1^*](t_{c})=T-t_{c}$ if $t_{c}\in (t_1^*,t_{f})$. Note that if $t_{c}=t_{1}^{*}$, $G(t_2)=T-t_2$ and if $t_{c}=t_{f}$, $G(t_2)=F[t_{1}^{*}](t_2)$. 
Then for $t_{1}^{*}$ fixed, the region ${\mathcal R}$ defined in \eqref{def: R} can be set as follows (see Figure \ref{fig: regionesR})
\begin{align}
 {\mathcal R}[t_{1}^{*}]&=\left\{ (t_2,\mu)\in \mathbb{R}^2:  t_{1}^{*}\le t_2\le t_{f}, \,  0\le \mu \le F[t_{1}^{*}](t_2), \, \mu+t_2\le T)\right\} \nonumber \\
 \label{eq: regionR} &=\left\{ (t_2,\mu)\in \mathbb{R}^2:  t_2\in [t_{1}^{*},t_f], \,  0\le \mu \le G(t_2) \right\}.
\end{align}
\begin{figure}[H]
\begin{center}
\includegraphics[width=4.2cm]{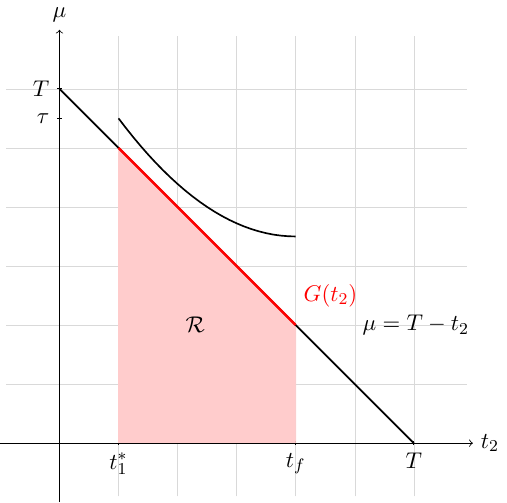}
\includegraphics[width=4.2cm]{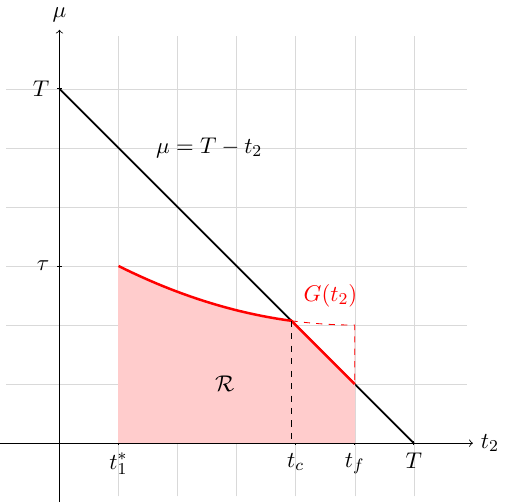}
\includegraphics[width=4.2cm]{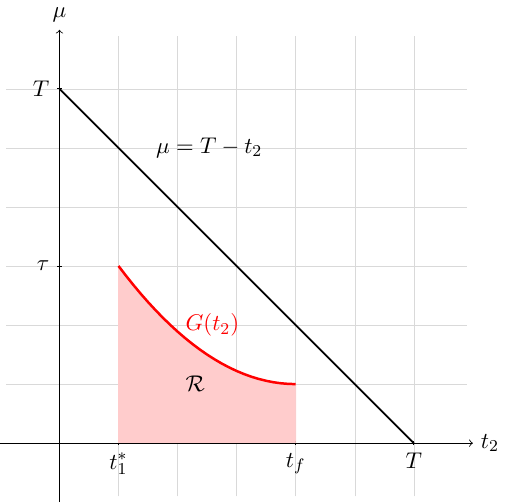}
\caption{{\bf Possible regions ${\mathcal R}$}. For the region on the left,  $F[t_{1}^{*}](t_{2}) > T-t_2 \,\forall t_2 \in [t_1^*,t_f]$, that is $t_c=t_1^*$. For the region in the center, there exists $t_c\in (t_1^*,t_f)$ such that $F[t_{1}^{*}](t_{c}) =T-t_c$. Finally, for the region on the right, $F[t_{1}^{*}](t_{2}) < T-t_2 \,\forall t_2 \in [t_1^*,t_f]$, that is $t_c=t_f$. }
\label{fig: regionesR}
\end{center}
\end{figure}

 Thus, we need to determine the maximum of the function
 \begin{equation} \label{eq: Jt2}
J(t_2,\mu)=x_{\infty}(x[{t_1^*,t_2,\mu}](T),y[{t_1^*,t_2,\mu}](T),\sigma_f) 
\end{equation}
with $(t_2,\mu)\in {\mathcal R}[t_1^*]$.
In order to do that, we begin by computing the partial derivative of $J(t_2,\mu)$ with respect to $\mu$ and obtain
 (see equation \eqref{eq: deriv_J_resp_mu} 
 in the supplement)

\begin{align} \label{eq: deriv_J_resp_mu_cuerpo}
\frac{\partial J}{\partial \mu}(t_2, \mu) 
&= 
\gamma  y_3 \frac{ x_{\infty}[t_1^*,t_2,\mu] }{1- \sigma_f x_{\infty}[t_1^*,t_2,\mu]} (\sigma_f - \sigma_s) >0, 
 \end{align}
where $x_{\infty}[t_1^*,t_2,\mu]=x_{\infty}(x[t_1^*,t_2,\mu)](T),y[t_1^*,t_2,\mu](T),\sigma_f)$. 

From \eqref{eq: deriv_J_resp_mu_cuerpo} we have that the optimal pair $(t_2^*,\mu^*)$ is attained at the superior border of the region ${\mathcal R}[t_1^*]$, that is, $\mu^*=G(t_2^*)$ with $t_2^*\in [t_1^*,t_{f}]$ and $G$ defined in \eqref{eq: Gpartida}. Thus we consider $J$ restricted to that border and define
\begin{align} \label{eq: Jtilde}
\tilde{J}(t_2):=J(t_2,G(t_2)).
\end{align}

From \eqref{eq: deriv_J_resp_t2_parte1} and  \eqref{eq: deriv_J_resp_t2_parte2}  in the Supplement we obtain
 \begin{equation} \label{eq: Jtilde_primat2}
\tilde{J}^{\prime} (t_2)= \frac{\gamma^2 y_3 K (1-\sigma_s x_2) x_{\infty}[t_1^*,t_2,G(t_2)] }{x_2 (1-\sigma_f x_{\infty}[t_1^*,t_2,G(t_2)])}. \left\{ \begin{array}{lll}
  w_b(t_2) & \text{ for } & t_2\in [t_1^*,t_{c}) \\
w_b(t_2) -\frac{1}{\gamma K}  & \text{ for } & t_2\in (t_{c},t_f]
\end{array}\right.
\end{equation}
where $x_2=x[t_1^*,t_2,G(t_2)](t_2)$, $y_3=y[t_1^*,t_2,G(t_2)](t_2+G(t_2))$ and $w_b:[t_1^*,t_{f}]\to\mathbb{R}$ is the continuous function 
 \begin{equation} \label{eq: wb}
w_b(t_2)= \int_{t_2}^{t_2+G(t_2)} \frac{\sigma_f x[t_1^*,t_2,G(t_2)](r)-1}{y[t_1^*,t_2,G(t_2)](r)} dr. \end{equation}


In what follows we analyse some properties of  the function $w_b$ on the interval $[t_1^*,t_{c}]$.
\begin{remark} \label{re: Gdecrece} 
Since $x[t_1^*,t_2,F[t_1^{*}](t_2)] (t_2+F[t_1^*](t_2))$ is decreasing as a function of $t_2$  defined on the interval $[t_1^*,t_{c}]$ (see Lemma \ref{le: varphi_decrece} in supplement),  we can define $s_0\in [t_1^*,t_{c}]$ such that
$x[t_1^*,t_2,F[t_1^*](t_2)](t_2+F[t_1^*](t_2))>1/\sigma_f$ for all $t_2\in[t_1^*,s_0)$ and $x[t_1^*,t_2,F[t_1^*](t_2)](t_2+F[t_1^*](t_2))<1/\sigma_f$ for all $t_2\in (s_0,t_{c}]$.
Note that if $x[t_1^*,t_2,F[t_1^*](t_2)](t_2+F[t_1^*](t_2))<1/\sigma_f$ for all $t_2\in (t_1^*,t_{c}]$, then $s_0=t_{1}^*$ and if $x[t_1^*,t_2,F[t_1^*](t_2)](t_2+F[t_1^*](t_2))>1/\sigma_f$ for all $t_2\in [t_1^*,t_{c})$, then $s_0=t_{c}$ .
 Then we can conclude that for $t_2 \in[t_1^*,s_0)$,  $x[t_1^*,t_2,F[t_1^*](t_2)](t_2+F[t_1^*](t_2))>1/\sigma_f$ and consequently $x[t_1^*,t_2,F[t_1^*](t_2)](r)\ge 1/\sigma_f$ for all $r\in (t_2,t_2+F[t_1^*](t_2))$. See Figure \ref{fig: phase_diagram_front}.
\begin{figure}[h]
  \begin{center}
    \includegraphics[width=9cm]{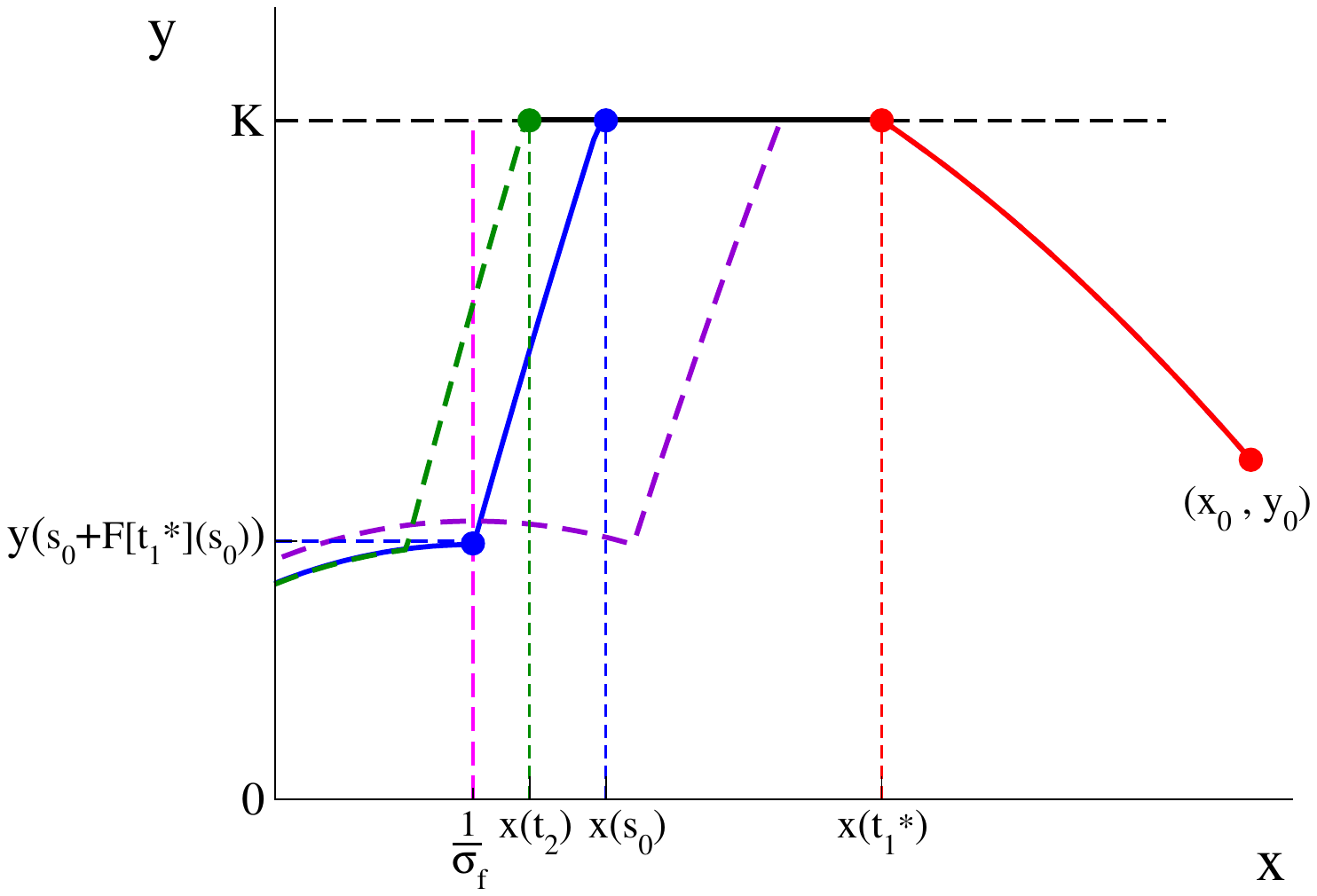}
    \caption{Phase diagram of the trajectory $(x[t_1^*,t_2,F[t_1^{*}](t_2)],y[t_1^*,t_2,F[t_1^{*}](t_2)])(r)$ in red for $r\in [t_0,t_1^*]$ and in black for $r\in [t_1^*,t_2]$.  The solid blue line for $r\in [s_0,s_0+F[t_1^{*}](s_0)]$ corresponds to $t_2=s_0$, while the dashed green and violet lines for $r\in [t_2,t_2+F[t_1^{*}](t_2)]$ correspond to $t_2> s_0$ and $t_2<s_0$ respectively.}
\label{fig: phase_diagram_front}
\end{center}
\end{figure}

 \end{remark}


\begin{lemma} \label{le: wbdecrece}
For $s_0 \in [t_1^*,t_{c}]$ given by remark \ref{re: Gdecrece},  $w_{b}(t_2)>0$ for all $t_2\in [t_1^*,s_0)$ and $w_b$ is a strictly decreasing function on $(s_0,t_{c})$ concluding that $w_b$ changes sign at most once on $[t_1^*,t_{c}]$. 
\end{lemma}
\begin{proof}
For $t_2\in [t_1^*,t_{c}]$, from \eqref{def: tc} and \eqref{eq: wb} we have that 
\begin{align} \label{eq: z}
w_{b}(t_2)=\int_{t_2}^{t_2+F[t_1^*](t_2)} \frac{\sigma_f x[t_1^*,t_2,F[t_1^*](t_2)](r)-1}{y[t_1^*,t_2,F[t_1^*](t_2)](r)} dr.
\end{align}
From Remark \ref{re: Gdecrece} $w_b(t_2)>0$ for all $t_2\in [t_1^*,s_0)$.
Since $t_1^*$ is fixed, given $t_2\in [t_1^*,t_{c}]$, the duration of the strict quarantine is given by $\mu={F[t_1^*](t_2)}$. For simplicity of notation in this proof we neglect the parameters $t_1^*$ and $F[t_1^*](t_2)$ from solutions $x[t_1^*,t_2,F[t_1^*](t_2)]$ and $y[t_1^*,t_2,F[t_1^*](t_2)]$ and write $x[t_2],y[t_2]$.
Given $t_2\in (s_0,t_{c})$, we define the auxiliary functions for $s\in(t_2,t_2+F[t_1^*](t_2))$
\begin{subequations}  \label{eq: func_aux}
\begin{align}
g[t_2](s)&=\sigma_f \sigma_s x[t_2](s) y[t_2](s)+(\sigma_f x[t_2](s)-1)(\sigma_s x[t_2](s)-1),  \label{eq: g} \\
f[t_2](s)&=\frac{\sigma_f x[t_2](s)-1}{y[t_2](s)}, \label{eq: f}\\
i[t_2](s)&=\frac{x[t_2](s)}{y[t_2](s)}(\sigma_f x[t_2](s)+\sigma_f y[t_2](s)-1)+\gamma  g[t_2](s) \int_t^s \frac{x[t_2](r)}{y[t_2](r)}dr. \label{eq: i}
\end{align}
\end{subequations}

By computing the derivatives for $s\in(t_2,t_2+F[t_1^*](t_2))$ we obtain
\begin{subequations}  \label{eq: deriv_func_aux}
{\small \begin{align}
\frac{d}{ds}g[t_2](s)&=-\gamma \sigma_s^2 x[t_2](s) y[t_2](s)(\sigma_f x[t_2](s)+\sigma_f y[t_2](s)-1), \label{eq: derivg} \\
\frac{d}{ds}f[t_2](s)&=-\gamma \frac{g[t_2](s)}{y[t_2](s)}, \label{eq: derivf}\\
\frac{d}{ds}i[t_2](s)&=-\gamma \sigma_s x[t_2](s)\left(\gamma \sigma_s   y[t_2](s) \int_t^s \frac{x[t_2](r)}{y[t_2](r)}dr +1\right)(\sigma_f x[t_2](s)+\sigma_f y[t_2](s)-1),\label{eq: derivi}
\end{align}}
\end{subequations}
and from \eqref{eq: deriv_x_y_t2_t2_t2masmu} and \eqref{eq: u_ti}-\eqref{eq: v_ti} with $i=2$ we have that
\begin{align}  \label{eq: z_prima2}
w_{b}^{\prime}(t_2)=&f[t_2]({t_2+F[t_1^*](t_2)}) - f[t_2](t_2) \\
&+ \gamma y_2 (\sigma_s -\frac{1}{x_2})  \int_{t_2}^{t_2+F[t_1^*](t_2)} \frac{x[t_2](s)}{ y[t_2]^2(s)}(\sigma_f x[t_2](s)+\sigma_f y[t_2](s)-1)ds \nonumber\\
&+ \gamma y_2 (\sigma_s -\frac{1}{x_2}) \int_{t_2}^{t_2+F[t_1^*](t_2)}  \gamma \left( \int_{t}^s \frac{x[t_2](r)}{y[t_2](r)} dr\right)  \frac{g[t_2](s)}{y[t_2](s)} ds \nonumber \\
= &f[t_2](t_2+F[t_1^*](t_2)) - f[t_2](t_2)+  \gamma y_2 (\sigma_s -\frac{1}{x_2}) \int_{t_2}^{t_2+F[t_1^*](t_2)}  \frac{i[t_2](s)}{ y[t_2](s)} ds  \nonumber
\end{align}

Note that if $g[t_2]$ is a positive function, then from \eqref{eq: derivf} $f[t_2]$ is decreasing. In addition, if $i[t_2]$ is also a positive function, we can conclude that $w_b$ is decreasing on $[t_1^*,t_{c}]$.
In what follows we will show that, in fact, $g[t_2](s)>0$ and $i[t_2](s)>0$ for all $s\in (t_2,{t_2+F[t_1^*](t_2)})$.

Since both derivatives of $g[t_2]$ and $i[t_2]$ have the opposite sign of $(\sigma_f x[t_2](s)+\sigma_f
y[t_2](s)-1)$, we need to analyse the latter function. First, assume $\sigma_f x[t_2](s)+\sigma_f y[t_2](s)> 1$ for all $s\in (t_2,{t_2+F[t_1^*](t_2)})$, then
from  \eqref{eq: derivg}, $g[t_2]$ is a decreasing function. Moreover, from Remark \ref{re: Gdecrece} we have that $\sigma_f x[t_2](t_2+F[t_1^*](t_2))-1<0$ for $t_2\in (s_0,t_{c})$ which combined with the inequality $\sigma_s
x[t_2](t_2+F[t_1^*](t_2))-1<0$ yields $g[t_2](s)>g[t_2](t_2+F[t_1^*](t_2))>0$ for all $s\in (t_2,{t_2+F[t_1^*](t_2)})$.  Thus from \eqref{eq: i} $i[t_2](s)>0$ for all $s\in (t_2,{t_2+F[t_1^*](t_2)})$.
\\
From the fact that $x+y$ is a decreasing function and $x(t_2)+y(t_2)>1/\sigma_f$, if we assume that there exists
$s_3\in [t_2,t_2+F[t_1^*](t_2)]$ such that $\sigma_f x[t_2](s_3)+\sigma_f y[t_2](s_3)=1$, from
\eqref{eq: derivg} and \eqref{eq: derivi} 
we have that
$g[t_2]$ and $i[t_2]$ attains a global minimum 
at $s_3$. Thus $g[t_2](s)\ge g[t_2](s_3)=\sigma_f y[t_2](s_3)>0$ and $i[t_2](s)\ge i[t_2](s_3)=\gamma g[t_2](s_3)\int_{t}^{s_3} \frac{x[t_2](r)}{y[t_2](r)} dr>0$ for all $s\in [t_2,{t_2+F[t_1^*](t_2)}]$ concluding also in this case that $g[t_2]$ and $i[t_2]$ are positive functions. 

Finally, from remark \ref{re: Gdecrece} we deduce that  $w_b$ changes sign at most once on $[t_1^*, t_{c}] $.

\end{proof}

\begin{lemma} \label{le: hachedecrece}
On the interval $[t_{c},t_{f}]$ the function $(w_b(t_2)-\frac{1}{\gamma K})y[t_1^*,t_2,T-t_2](T)$ is decreasing with respect to $t_2$. Furthermore $w_b(t_f)\le 0$ and $w_b(t_1^*)>0$. \end{lemma}

\begin{proof}
For $t_2\in [t_{c},t_f]$, $y_3= y[t_1^*,t_2,T-t_2](T)$, from equation \eqref{eq: cuenta_aux} we
have the identity
\begin{align} \label{eq: equivhache}
 \left(w_b(t_2)- \frac{1}{\gamma K}\right) y_3
=   (\sigma_f-\sigma_s) y_3  \int_{t_2}^T \frac{ x[t_1^*,t_2,T-t_2](r)}{y[t_1^*,t_2,T-t_2](r)} dr-\frac{1}{\gamma }.
\end{align} 
Using that for $r\in [t_2,T]$
\begin{align*}
\frac{\partial}{\partial t_2}\left(  \frac{y[t_1^*,t_2,T-t_2](T)}{y[t_1^*,t_2,T-t_2](r)} \right) =\frac{y[t_1^*,t_2,T-t_2](T)}{y[t_1^*,t_2,T-t_2](r)} \gamma\sigma_s \int_r^T \frac{\partial}{\partial t_2} x[t_1^*,t_2,T-t_2](r) dr,
\end{align*}
equations \eqref{eq: deriv_x_mu_0_t2masmu}, \eqref{eq: deriv_x_t2_t2_t2masmu}, the definition of $u_2$ given in \eqref{eq: def_u} and its positivity (see \eqref{eq: u_ti}) we obtain that the right hand side of  equation \eqref{eq: equivhache} is decreasing as a function of $t_2$ yielding the left hand side is also decreasing as a function of $t_2$.

The second statement follows from the definition of $t_f$ from where we have that $G(t_f)=0$ or $x(r)<\frac{1}{\sigma_f}$ for all $r\in [t_f,t_f+G(t_f)]$ implying that $w_b(t_f)\le 0$. Also, assume $w_{b}(t_1^*)\le 0$, then from Lemma \ref{le: wbdecrece} $s_0=t_1^*$ and $w_b$ is strictly decreasing on $(t_1^*,t_{c})$ from where $w_b(t_2) < 0$ for all $t_2\in(t_1^*,t_{c}]$. Moreover, using that $w_b(t_{c})\le 0$, the positivity of $\gamma K$ and $y[t_1^*,t_{c},T-t_{c}](T)$ and the monotonicity of the function $(w_b(t_2)-\frac{1}{\gamma K})y[t_1^*,t_2,T-t_2](T)$ with respect to $t_2$ in the interval $[t_{c},t_{f}]$ we obtain that $y[t_1^*,t_2,T-t_2](T)\left( w_b(t_2)-\frac{1}{\gamma K}\right) <0$ for all $t_2\in [t_{c},t_f]$. Thus, from \eqref{eq: Jtilde_primat2}, we conclude that $\tilde{J}(t_2)$ is decreasing on $(t_1^*,t_f)$ yielding $t_2^*=t_1^*$ and $\mu^*=G(t_1^*)$ contradicting the existence of a boundary arc (Lemma \ref{le: extremal restringido tiene arco frontera}).  

\end{proof}

\begin{theorem} \label{te: control_optimo_sigma1gral_rest_v2}
Let $0\le \sigma_s<\sigma_f$ with $\sigma_s<1$. Let $\sigma^0$ be the optimal control for the unconstrained problem  \eqref{eq: funcional_J_iso}-\eqref{eq: rest_int} given by \eqref{eq: sigmacero_con_rest_int} and $(x^0,y^0)$ the associated optimal state given by the solution of the system \eqref{eq: SIR_iso_x}-\eqref{eq: SIR_iso_y} with $\sigma=\sigma^0$. 
In the case that there exists $t\in [0,T]$ such that $y^0(t)>K$ we assume that any contact or exit point $s$ of the constraint extremal satisfies $x^*(s)>\frac{1}{\sigma_f}$. Then the optimal control $\sigma^*$ for the restricted problem \eqref{eq: funcional_J_iso}-\eqref{eq: restr_yt} is unique and is given by
\begin{equation} \label{eq: sigma_sin_rest}
\sigma^*(t):=\sigma[t_1^*,t_2^*,\mu^*](t)= \left\{ \begin{array}{ll}
\sigma_f& \text{ for } 0\le t \le t_1^*\\
\sigma_b(t) & \text{ for } t_1^* <t \le t_2^*\\
\sigma_s & \text{ for } t_2^* < t  \le t_2^* +\mu^* \\
\sigma_f& \text{ for } t_2^* +\mu^* \le t \le T,
\end{array}\right.
\end{equation}
where $\sigma_b(t)=1/x^*(t)$. Moreover,

{\bf Case 1}: If  $y^0(t)\le K$ for all $t\in [0,T]$, then $t_1^*=t_2^*=t^0$, $\mu^*=\mu^0$ obtaining the result given in Theorem \ref{te: control_optimo_sigma1gral} where $\sigma[t^0,t^0,\mu^0]=\sigma^0$.

{\bf Case 2}: If there exists $t\in [0,T]$ such that $y^0(t)>K$, then $t_1^*= t_{b}$ and for $t_{c}$ defined in \eqref{def: tc}  and $w_b$ defined in \eqref{eq: wb} we have
\begin{itemize}
\item[2.1.] If $w_{b}(t_{c})\le 0$, then there exists a unique $\bar{t} \in (t_1^*,t_{c})$ such that $w_{b}(\bar{t})=0$ and thus
$t_2^*=\bar{t}$, $\mu^*=G(\bar{t})$. 
\item[2.2.] If $\displaystyle 0< w_{b}(t_{c})\le \frac{1}{\gamma K}$, then $t_2^*=t_{c}$, $\mu^*=G(t_{c})$. 
\item[2.3.] If $\displaystyle w_{b}(t_{c})> \frac{1}{\gamma K}$, then there exists a unique 
$\breve{t} \in (t_{c},t_f)$ such that $w_{b}(\breve{t})=\frac{1}{\gamma K}$ and thus
$t_2^*=\breve{t}$, $\mu^*=G(\breve{t})$.
\end{itemize}
\end{theorem}

\begin{proof}

Let $\sigma^0$ and $(x^0,y^0)$ be as in Theorem \ref{te: control_optimo_sigma1gral} and $\sigma^*$ and $(x^*,y^*,v^*)$ an optimal solution of the restricted problem. 
We consider the following cases:

{\bf Case 1}: If $y^0(t)\le K$ for all $t\in [0,T]$, it satisfies constraint \eqref{eq: restr_yt}. Therefore $\sigma^0$ is an admissible control,  concluding that is also optimal for the constrained problem. Thus, the optimal control consists of taking $t_2^*=t_1^*=t^0$ and $\mu^*=\mu^0$.

{\bf Case 2}: If the exists $t\in [0,T]$ such that $y^0(t)> K$, from lemma \ref{le: t1*_es_tb} we have that $t_1^*=t_b$.
In what follows we will characterise $(t_2^*,\mu^*) \in R[t_1^*]$ defined in \eqref{eq: regionR}.
 
\begin{itemize}
\item[{\it 2.1.}] Since $w_{b}(t_1^*)> 0$ (see Lemma \ref{le: hachedecrece}) and $w_{b}(t_{c})\le 0$, from Lemma \ref{le: wbdecrece} there exists a unique $\bar{t} \in (t_1^*,t_{c}]$ such that $w_b(\bar{t})=0$, $w_b(s)>0$ for $s\in [t_1^*,\bar{t})$ and $w_b(s)<0$ for $s\in (\bar{t},t_{c}]$. Thus, $\tilde{J}(t_2)$ increases on $[t_1^*,\bar{t})$ and, in the same way as for the previous item, from \eqref{eq: Jtilde_primat2}, we deduce that $\tilde{J}(t_2)$ decreases on $(\bar{t},t_f]$ proving that $t_2^*=\bar{t}$ and $\mu^*=G(\bar{t})$.
\item[{\it 2.2.}] If $\displaystyle 0< w_{b}(t_{c})\le \frac{1}{\gamma K}$, from Lemma \ref{le: wbdecrece} we have that $w_b(t_2)> 0$ for all $t_2\in [t_1^*,t_{c}]$ and then,  from \eqref{eq: Jtilde_primat2}, $\tilde{J}(t_2)$ is increasing on $[t_1^*,t_{c}]$.  On the other hand, since 
$y[t_1^*,t_{c},T-t_{c}](T)\left(w_{b}(t_{c})- \frac{1}{\gamma K} \right)\le 0$, from Lemma \ref{le: hachedecrece}  and equation \eqref{eq: Jtilde_primat2}, we have that  $\tilde{J}(t_2)$ is decreasing on $[t_{c},t_f]$ concluding that $t_2^*=t_{c}$ and $\mu^*=G(t_{c})$.
\item[{\it 2.3.}]
If $\displaystyle w_{b}(t_{c})> \frac{1}{\gamma K}$, from Lemma \ref{le: wbdecrece} we have that $w_b(t_2)> 0$ for all $t_2\in [t_1^*,t_{c}]$. On the other hand, since $y[t_1^*,t_{c},T-t_{c}](T)\left(w_{b}(t_{c})- \frac{1}{\gamma K} \right) > 0$ from 
Lemma \eqref{le: hachedecrece}, we have that there exists a unique $\breve{t}  \in (t_{c},t_f)$ such that  $w_b(\breve{t})=\frac{1}{\gamma K}$.
Thus, from \eqref{eq: Jtilde_primat2}, we have that  $\tilde{J}(t_2)$ is increasing on $[t_1^*,\breve{t}]$ and decreasing on $[\breve{t},t_{f}]$ proving that  $t_2^*=\breve{t}$ and $\mu^*=G(\breve{t})$. 
\end{itemize}
\end{proof}

\begin{remark}\label{re: final}
From the definition of $t_c$ we have three possibilities: $t_c=t_1^*$, $t_c \in (t_1^*,t_f)$ or $t_c=t_f$.
In the first case, when $t_c=t_1^*$, item  2.1. in Theorem \ref{te: control_optimo_sigma1gral_rest_v2} is never satisfied. Moreover, using a similar argument as in the proof of Lemma \ref{le: hachedecrece}, from Lemma \ref{le: extremal restringido tiene arco frontera} item 2.2. cannot be fulfilled obtaining $t_2^*=\breve{t}$ and $\mu^*=G(\breve{t})=T-\breve{t}$. 
If $t_c \in (t_1^*,t_f)$, the three items can be satisfied, for item 2.1. in Theorem  \ref{te: control_optimo_sigma1gral_rest_v2} $G(t_2^*)=F[t_1^*](t_2^*)$ and for items 2.2. and 2.3. $G(t_2^*)=T-t_2^*$.
 Finally, in the third case, when $t_c=t_f$, from Lemma \ref{le: hachedecrece} we have that items 2.2. and 2.3. in Theorem  \ref{te: control_optimo_sigma1gral_rest_v2} are never satisfied and thus we have $t_2^*=\bar{t}$, $\mu^*=G(\bar{t})=F[t_1^*](\bar{t})$. 
\end{remark}


\section{Numerical results} 
\label{se: numerical}

Theorem~\ref{te: control_optimo_sigma1gral_rest_v2} summarises the main results obtained in this article, which provides a rigorous optimal solution of the control problem defined in equations~(\ref{eq: SIR_iso}).  That is, Theorem~\ref{te: control_optimo_sigma1gral_rest_v2} gives the optimal starting time of the strict quarantine $t_2^*$ and its associated duration $\mu^*$ that minimises the cumulative number of ever-infected individuals (recovered), for a given set of epidemic parameters and  under two restrictions: 
the maximum allowed fraction of infected individuals $K$ over the control period $T$ (\ref{eq: restr_yt}), and (\ref{eq: rest_int}) implying that the maximum duration of the strict quarantine for the optimal solution is $\tau$ (see Lemma \ref{le: F_crece}).
In particular, the optimal solutions can be grouped into two different case scenarios that depend on the relation between $K$ and $y_{max}^0$ (the maximum value of $y^0(t)$ in the interval $[0,T]$, where $y^0(t)$ is the optimum trajectory of the model without restriction). When $y_{max}^0$ is larger than $K$, the optimum $t_2^*$ is larger than the time $t_1^*$ at which the trajectory of the system in the $x$--$y$ space enters into the boundary arc $y(t)=K$, while for  $y_{max}^0$ smaller than $K$ the system never enters into the boundary arc and the optimum $t_2^*$ corresponds to that of the model with no restriction, $t_2^*=t^0$ and $\mu^*=\mu^0$ (Theorem~\ref{te: control_optimo_sigma1gral}).

In order to test Theorem~\ref{te: control_optimo_sigma1gral_rest_v2}, here we compare a numerical solution of the control problem (\ref{eq: SIR_iso}) with that predicted by Theorem~\ref{te: control_optimo_sigma1gral_rest_v2}.  We integrated the system of equations~(\ref{eq: SIR_iso_x}) and (\ref{eq: SIR_iso_y}) numerically under the constrains given by equations~(\ref{eq: rest_int}) and (\ref{eq: restr_yt}), and found approximate values for the optima $t_2^*$ and $\mu^*$.  Below we present results for two different thresholds, $K=0.03$ (Figs.~\ref{w-t2-mu-K-003} and \ref{sigma-x-y-K-003}) and $K=0.06$ (Figs.~\ref{w-t2-mu-K-006} and \ref{sigma-x-y-K-006}), which allow to capture the different regimes corresponding to all possible different cases described in Theorem~\ref{te: control_optimo_sigma1gral_rest_v2}.  For that, we calculated numerically $y_{max}^0$ and the functions $w$ and $w_b$, for several values of $\tau$ in the range $[0,155]$, while keeping the other parameters fixed ($T=365, \gamma=0.1, \sigma_s =0.8, \sigma_f=1.5$), as shown in panels (c) and (d) of Figs.~\ref{w-t2-mu-K-003} and \ref{w-t2-mu-K-006}.  Then, for a given $\tau$, the values of $y_{max}^0$, $w$ and $w_b$ determine the case as given in Theorem~\ref{te: control_optimo_sigma1gral_rest_v2}, and thus the numerical values of $t_2^*$ and $\mu^*$. 

Figure~\ref{w-t2-mu-K-003} shows cases $2.1$, $2.2$ and $2.3$ of Theorem~\ref{te: control_optimo_sigma1gral_rest_v2}  that are observed for a relatively low threshold $K=0.03$.  Here, the maximum value $y_{max}^0$ that takes $y^0(t)$ when there is no threshold restriction overcomes $K$ for all $\tau$ values ($y_{max}^0>K$).  These cases correspond to different intervals of $\tau$ determined by the functions $w_b(t_c)$ and $(\gamma K)^{-1}$, as we see in panels (c) and (d).  In respective panels (a) and (b) we compare the behaviour of $t_2^*$ and $\mu^*$ as a function of $\tau$ calculated numerically (filled circles) with that predicted from Theorem~\ref{te: control_optimo_sigma1gral_rest_v2} (empty squares), where we observe an excellent agreement.  The numerical estimation of the optimum (filled circles) was obtained by numerical integration of the dynamical system of equations (\ref{eq: SIR_iso_x}) and (\ref{eq: SIR_iso_y}) with initial condition $(x_0,y_0)=(1-10^{-6},10^{-6})$ and constraints (\ref{eq: rest_int}) and (\ref{eq: restr_yt}), for several sets of $t_2$ and $\mu$, and calculating the values $t_2^*$ and $\mu^*$ for which the asymptotic fraction of susceptible individuals $x_{\infty}$ reaches a maximum value.  For its part, the optimum from Theorem~\ref{te: control_optimo_sigma1gral_rest_v2} (empty squares) was calculated according to the numerically estimated values of $\overline{t}$, $t_c$ and $\breve{t}$ for $t_2^*$, and $G(\overline{t})=F[t_1^*](\overline{t})$, $G(t_c)=T-t_c$ and $G(\breve{t})=T-\breve{t}$ for $\mu^*$, which correspond to cases $2.1$, $2.2$ and $2.3$, respectively.

Right panels of Fig.~\ref{sigma-x-y-K-003} show the evolution of the system in the $x$--$y$ space with threshold $K=0.03$, for three different values of $\tau$ that correspond to the three different cases $2.1$, $2.2$ and $2.3$, and various values of $t_2$ (beginning of strict quarantine $\sigma(t)=\sigma_s$) in each case.  All curves $(x(t),y(t))$ start at $(x_0,y_0)$ and initially follow the curve with quarantine free reproduction number $\sigma(t)=\sigma_f$.  Then, some curves hit the boundary arc $y(t)=K$ at $t_1^* \simeq 212.93$, and then continue along it with $\sigma(t)=1/x(t)$ until they quit the boundary arc at either time $t_2$ when the strict quarantine begins, or $t_m$ when $x(t_m)=1/\sigma_f$.  One curve shows the case when the strict quarantine is applied before the boundary arc ($t_2=205<t_1^*$), and the rest of the cases correspond to applying the strict quarantine after $t_m$.  In each curve the strict quarantine lasts a time $\mu$ that depends on $t_2$ and fulfils the restriction $v(T) \ge \sigma_s \tau + \sigma_f (T-\tau)$.  The optimum $(t_2^*,\mu^*)$ leads to the maximum of $x_{\infty}$ (pink curves).
We see that in all cases the optimum corresponds to applying the strict quarantine after hitting the boundary arc ($t_2^*>t_1^*$) during a time period $\mu^*$ that is smaller than the maximum allowed $\tau$.  In other words, it turns more efficient to start the strict quarantine while on the boundary arc, and keep it for a time shorter than the maximum allowed $\tau$. 

The left panels of Fig.~\ref{sigma-x-y-K-003} show the associated evolution of $\sigma(t)$ for the optimum control $(t_2^*$,$\mu^*)$ (pink line), and for other two values $(t_2,\mu)$ close to the optimum, for each case.

Top panels of Fig.~\ref{sigma-x-y-K-003} show the case $2.1$ ($\tau=40$), where the maximum of $x_{\infty}$ is reached starting the strict quarantine at $t_2^*=\overline{t} \simeq 286.8$ for a period $\mu^*=F[t_1^*](\overline{t}) \simeq 16.5<\tau$.  Here we clearly see that implementing a shorter and later strict quarantine is more effective than using a longer but earlier strict quarantine, as we can see in the top--left panel by comparing $\sigma(t)$ for $\mu \simeq 20.5$ (green curve) and $\mu^* \simeq 16.5$ (pink curve).  Middle panels show the case $2.2$ ($\tau=110$), where the optimum corresponds to $t_2^*=t_c \simeq 277.8$ and $\mu^*=T-t_c \simeq 87.2<\tau$. Finally, bottom panels show the case $2.3$ ($\tau=140$), in which $t_2^*=\breve{t} \simeq 277.2$ and $\mu^*=T-\breve{t} \simeq 87.8<\tau$. 

In Fig.~\ref{w-t2-mu-K-006} we show cases $2.1$ and $1.2$, $1.3$ and $1.4$ of Theorem \ref{te: control_optimo_sigma1gral} that are seen for a high threshold $K=0.06$, where $y_{max}^0>K$ for $\tau \lesssim 70$ (case $2.1$), while $y_{max}^0<K$ for $\tau \gtrsim 70$ (cases $1.2$, $1.3$ and $1.4$).  In Fig.~\ref{w-t2-mu-K-006}(c) we observe that the region $\tau \lesssim 70$ corresponds to case $2.1$ where $w_b(t_c)<0$, and thus the optimum is given by $(t_2^*,\mu^*)=(\overline{t},G(\overline{t}))$.  Then, the region $70 \lesssim \tau \lesssim 123.75$ corresponds to case $1.2$ ($w(T-\tau)<0$) with an optimum at $(t_2^*,\mu^*)=(\hat{t},\tau)$, the region $123.75 \lesssim \tau \lesssim 124.93$ is the case $1.3$ ($0<w(T-\tau)\le (\gamma \, y(T-\tau))^{-1}$) with $(t_2^*,\mu^*)=(T-\tau,\tau)$, and the last region $\tau \gtrsim 123.93$ represents the case $1.4$ ($w(T-\tau) > (\gamma y(T-\tau))^{-1}$), where $(t_2^*,\mu^*)=(\tilde{t},T-\tilde{t})$.  Figure~\ref{sigma-x-y-K-006} is analogous to that of Fig.~\ref{sigma-x-y-K-003}, where we show the trajectories $(x,y)$ (right panels) and their associated evolution of $\sigma(t)$ (left panels) for three values of $\tau$ corresponding to cases $1.2$ ($\tau=80$), $1.3$ ($\tau=123.8$) and $1.4$ ($\tau=150$), with optima $(t_2^*,\mu^*)$ at $(\hat{t},\tau) \simeq (242.2,80)$, $(T-\tau,\tau) \simeq (241.2,123.8)$ and $(\tilde{t},T-\tilde{t}) \simeq (241.0,124.0)$, respectively.

\begin{figure}
  \begin{center}
    \includegraphics[width=9cm]{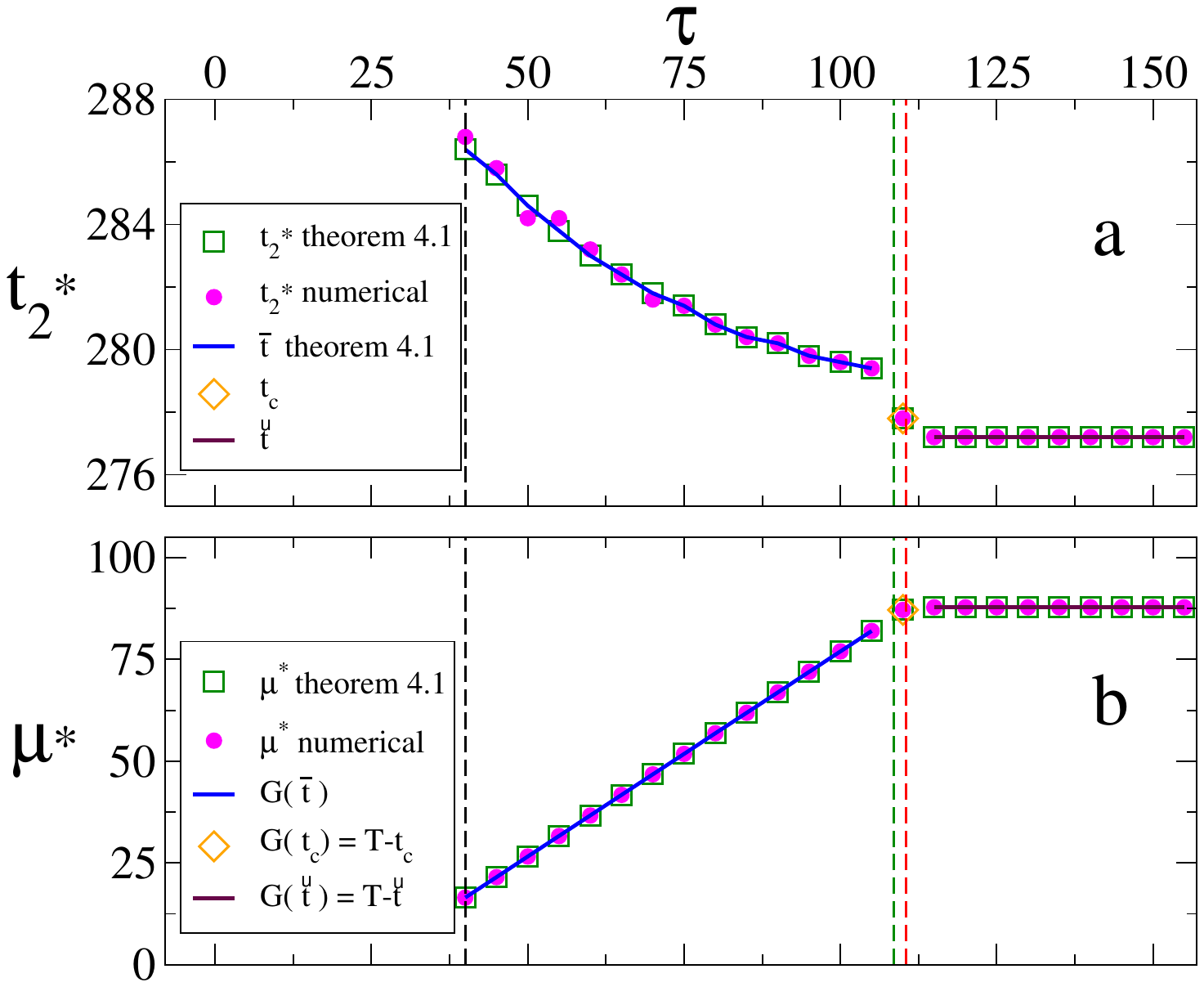} \\
    \vspace{-0.4cm}       
    \includegraphics[width=9cm]{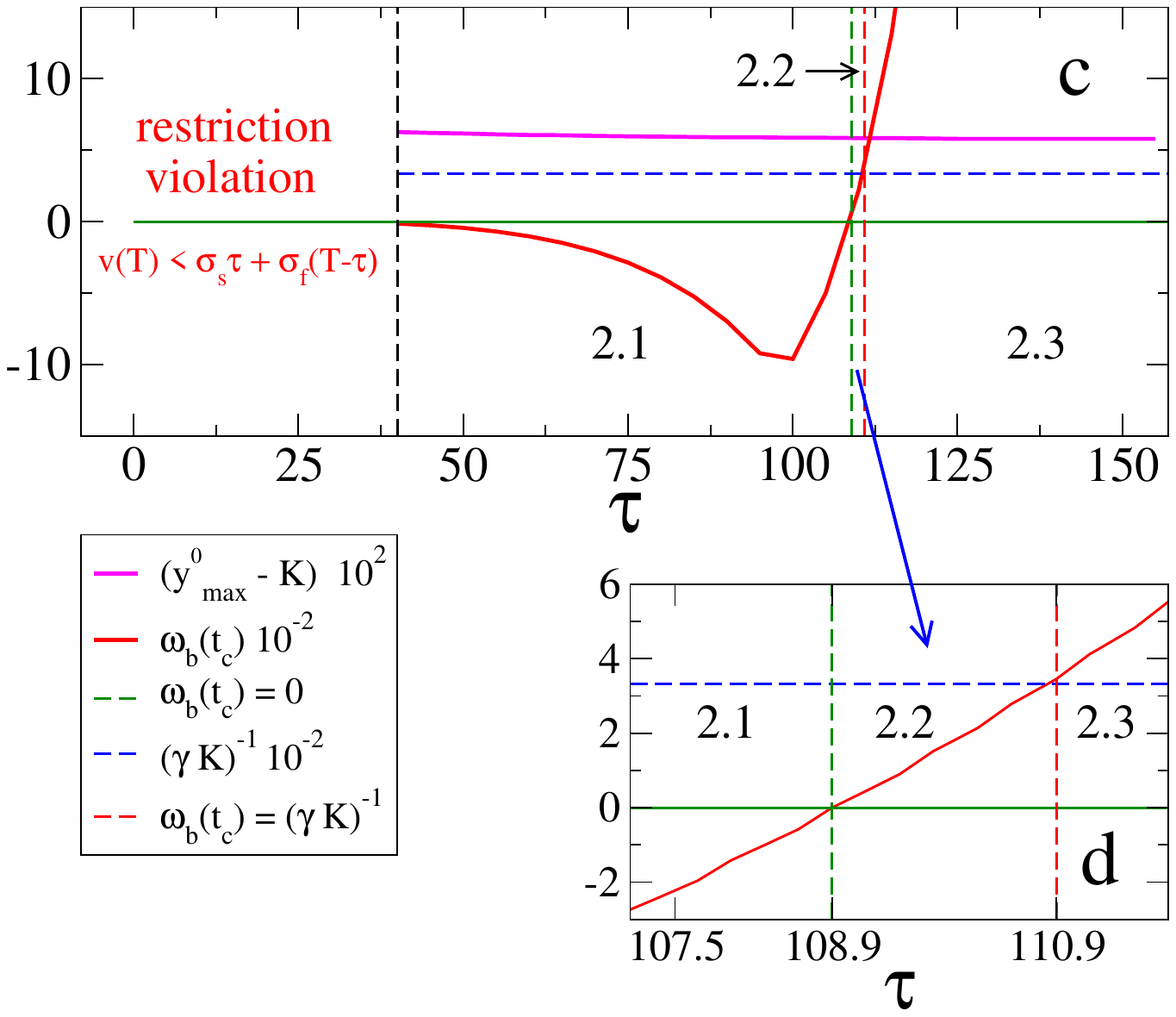}
  \end{center}  
  \caption{Optimal solution for threshold $K=0.03$, control period $T=365$, recovery rate $\gamma=0.1$, strict quarantine reproduction number $\sigma_s=0.8$, and quarantine free reproduction number $\sigma_f=1.5$.  The initial condition is $x_0=1-10^{-1}$ and $y_0=10^{-6}$.  (a) Optimal initial time of the strict quarantine $t_2^*$ vs maximum allowed duration of the strict quarantine $\tau$.  (b) Optimal duration of the strict quarantine $\mu^*$ vs $\tau$.  In both panels (a) and (b), empty squares represent the exact solution given by Theorem~\ref{te: control_optimo_sigma1gral_rest_v2}, solid circles correspond to the numerical solution that maximises the asymptotic fraction of susceptible individuals $x_{\infty}$, and solid lines are the exact solution in each region.  (c) Graphical determination of the times $\tau \simeq 108.9$ and $\tau \simeq 110.9$ denoted by vertical dashed lines that define the three regions for the different behaviours of $t_2^*$ and $\mu^*$.  (d) Same as panel (c) with a smaller scale to see region $2.2$ in more detail.}
  \label{w-t2-mu-K-003}
\end{figure}

\begin{figure}    
  \begin{center} 
    \includegraphics[width=7.5cm]{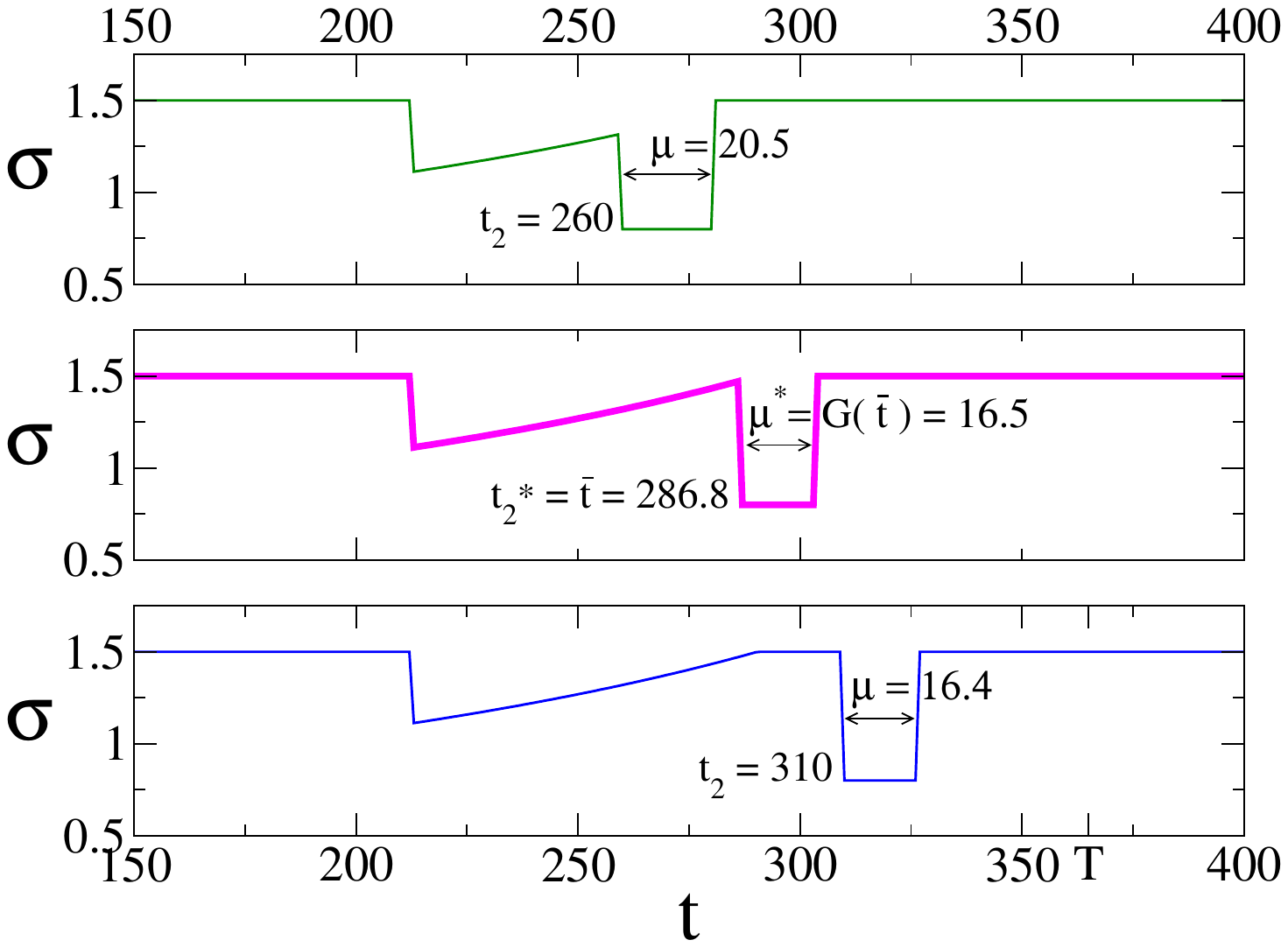} 
    \includegraphics[width=7.5cm]{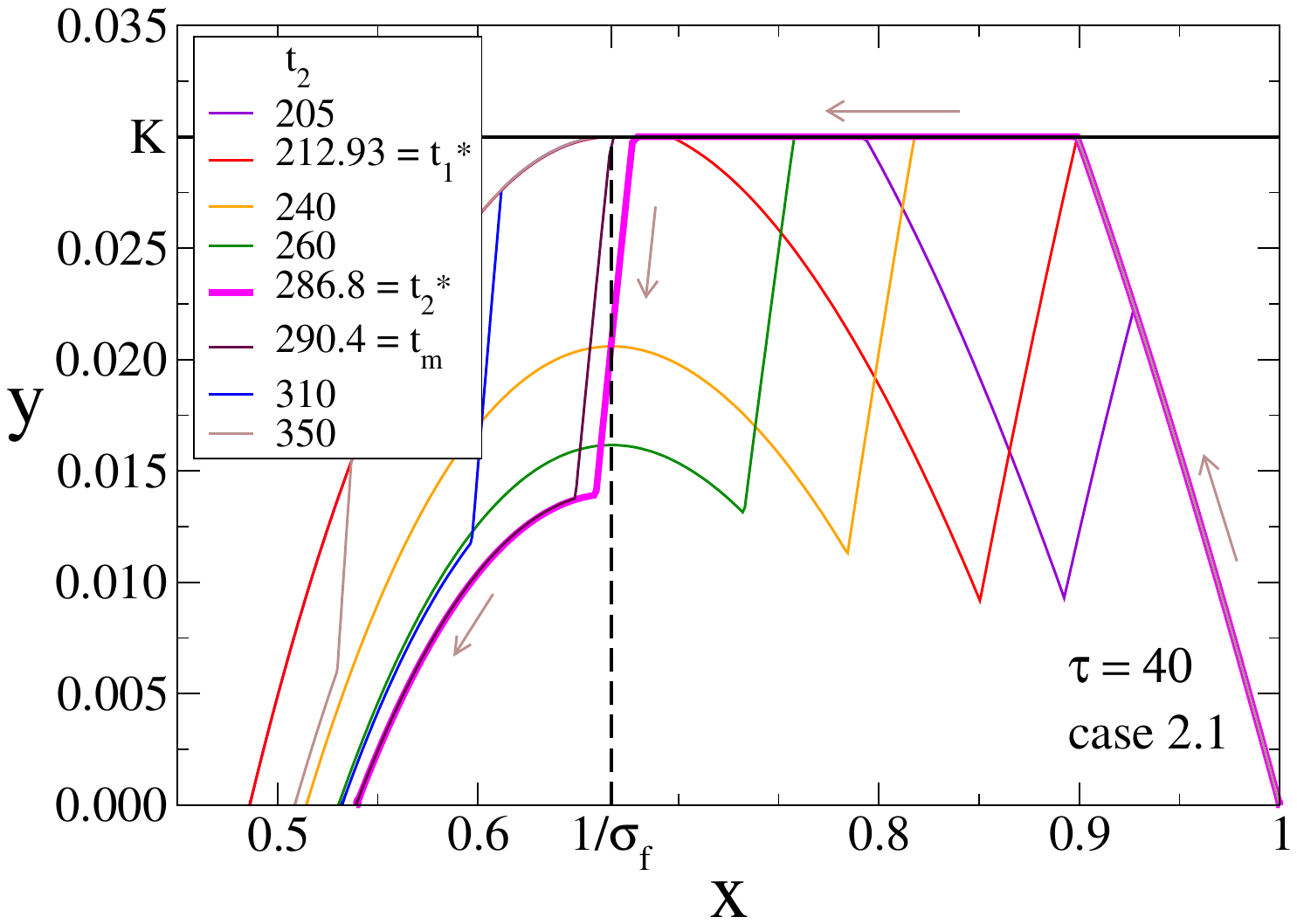} \\
    \vspace{-0.3cm}
    \includegraphics[width=7.5cm]{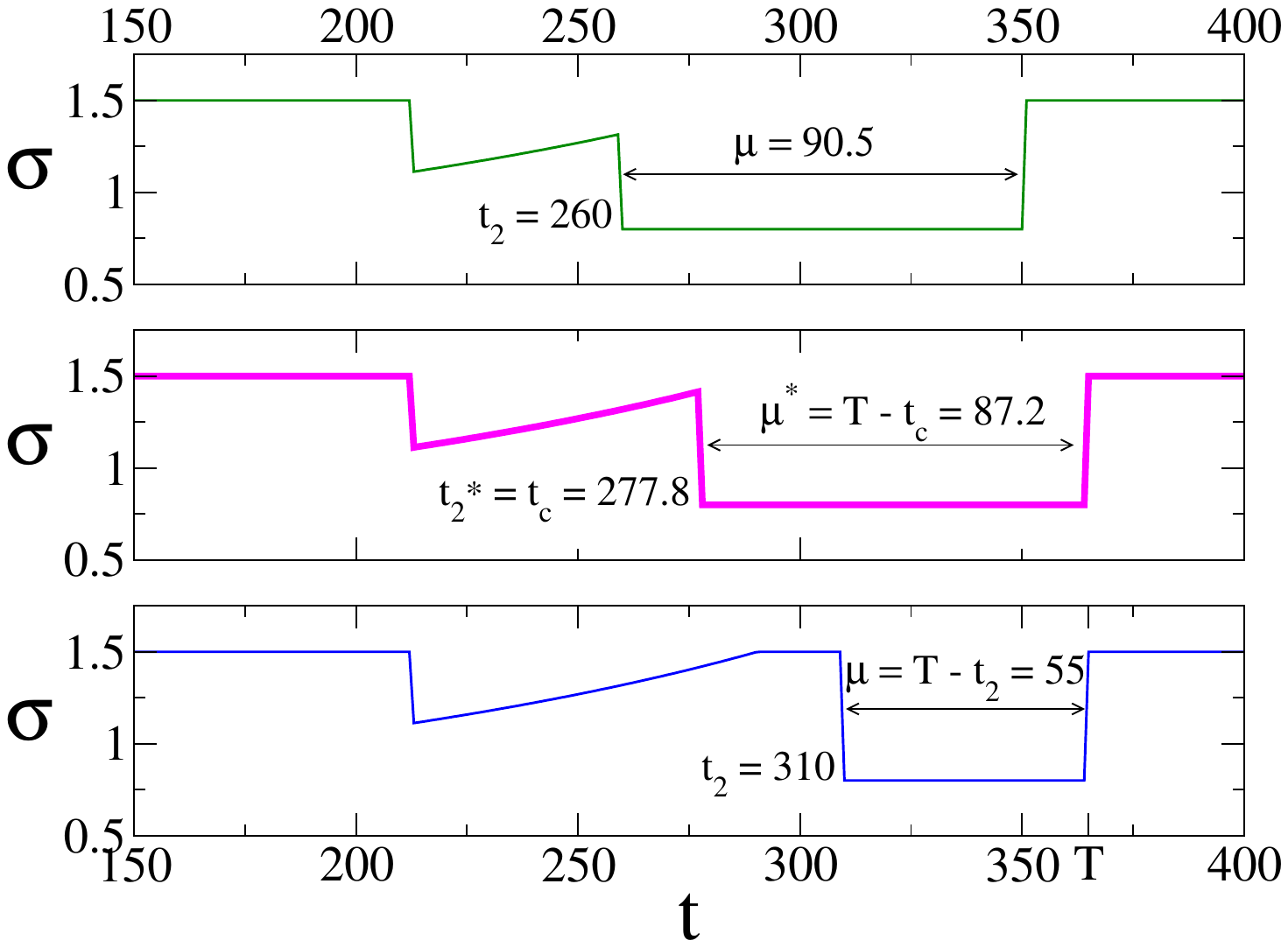} 
    \includegraphics[width=7.5cm]{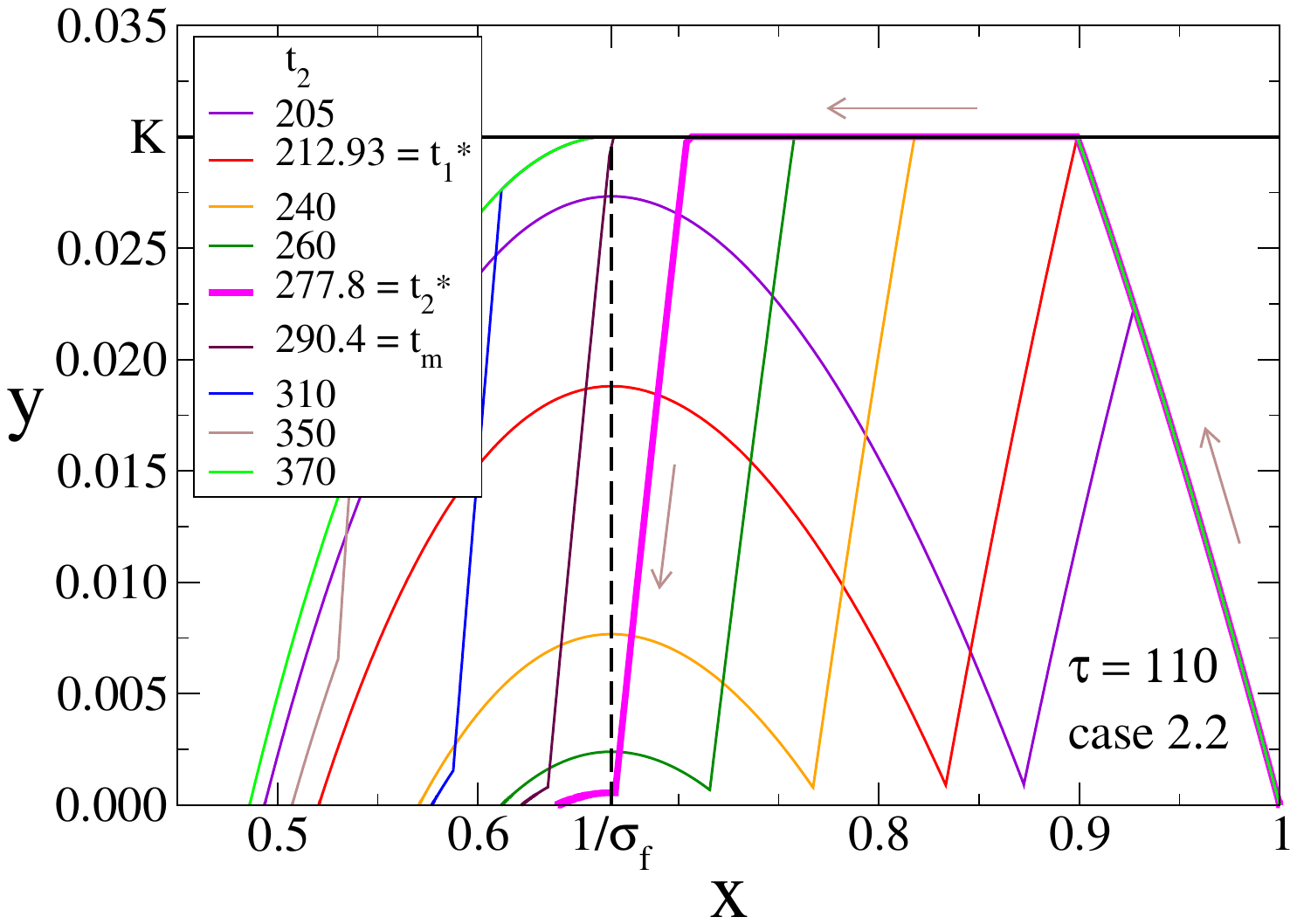} \\
    \vspace{-0.3cm}
    \includegraphics[width=7.5cm]{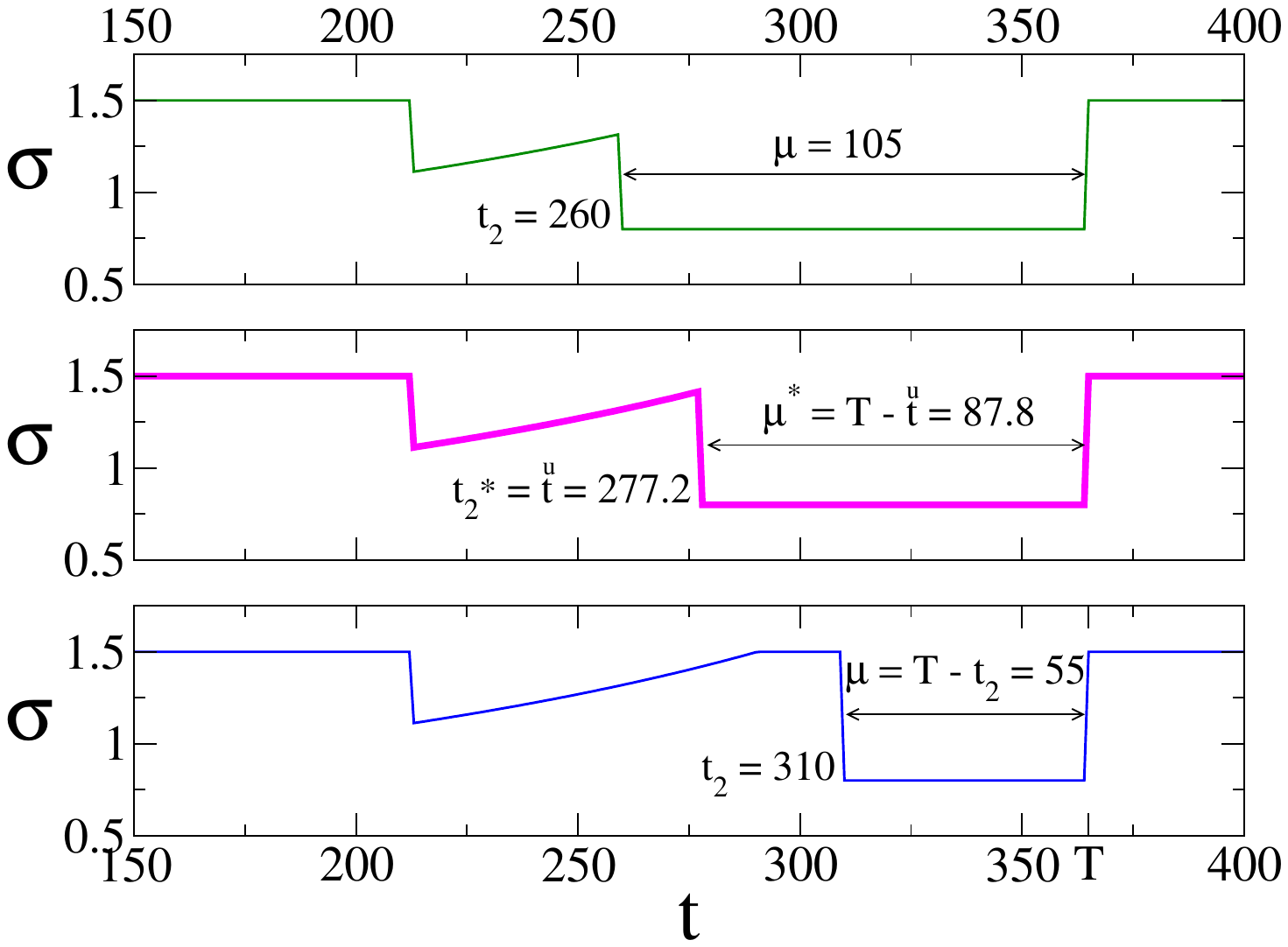}  
    \includegraphics[width=7.5cm]{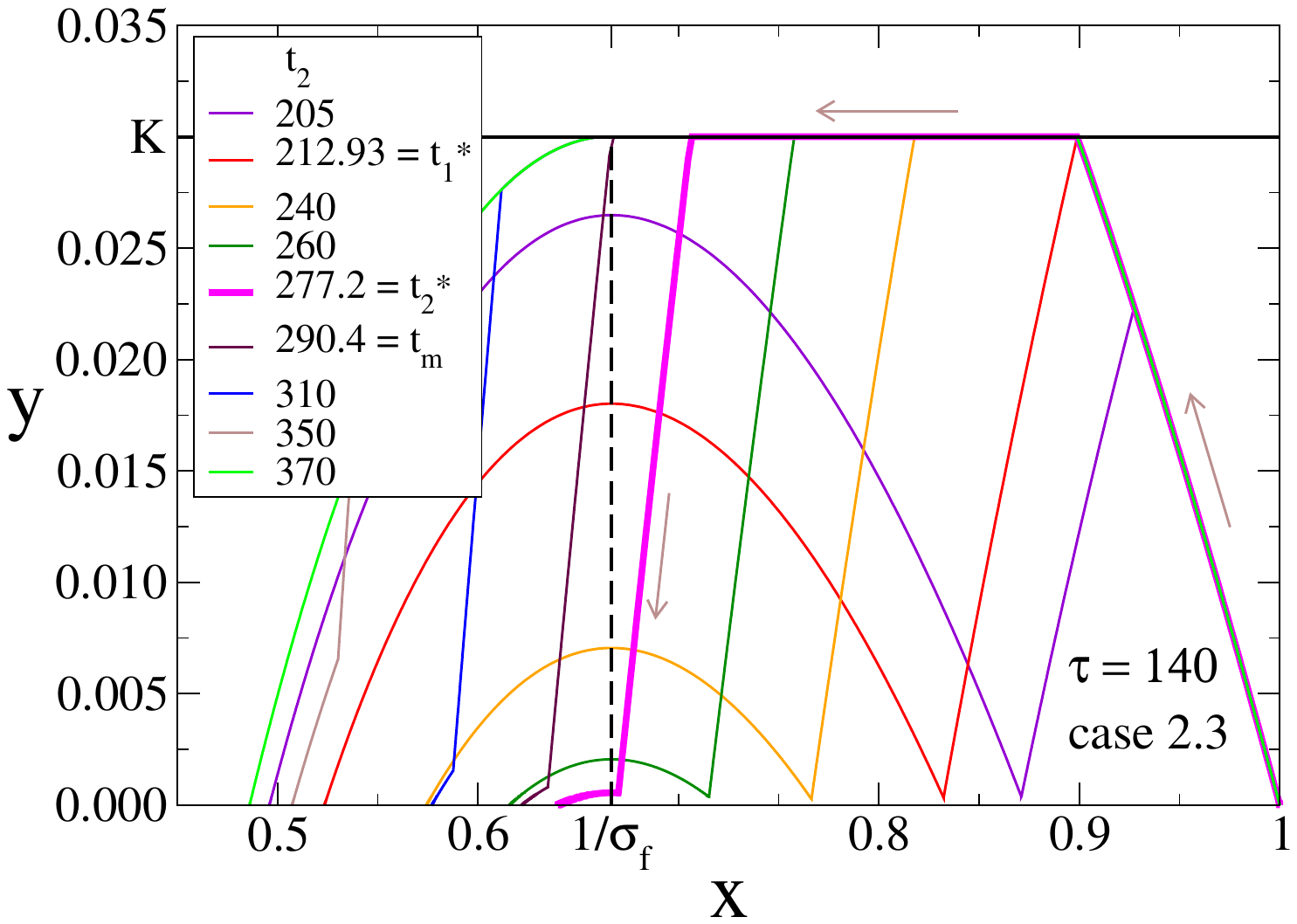} \\   
    \vspace{-0.8cm}
  \end{center}  
  \caption{System's trajectory in the $x$--$y$ phase space (right
    panels) for $K=0.03$, $T=365$, $\gamma=0.1$, $\sigma_s=0.8$ and $\sigma_f=1.5$, and the three values of $\tau$ indicated in the legends, corresponding to three different cases of the optimum time $t_2^*$, according to Theorem~\ref{te: control_optimo_sigma1gral_rest_v2}.  Each individual trajectory corresponds to a different initial time of the strict quarantine $t_2$.  Pink lines denote the optimum trajectory, where the strict quarantine starts at the optimum time $t_2^*$ after the hitting time $t_1^* \simeq 212.93$ for a period $\mu^*$ (minimum of $x_{\infty}$), while for the red and brown trajectories the strict quarantine starts when the system hits ($t_2=t_1^*$) and leaves ($t_2=t_m \simeq 290.4$) the boundary arc $y(t)=K$, respectively.  Left panels show the time evolution of $\sigma(t)$ for three different initial times $t_2$ in each case.  The optimum times are $t_2^* \simeq 286.8$ for $\tau=40$ (top panels), $t_2^* \simeq 277.8$ for $\tau=110$ (middle panels), and $t_2^* \simeq 277.2$ for $\tau=140$ (bottom panels).}
  \label{sigma-x-y-K-003}
\end{figure}

\begin{figure}
  \begin{center}
    \includegraphics[width=9cm]{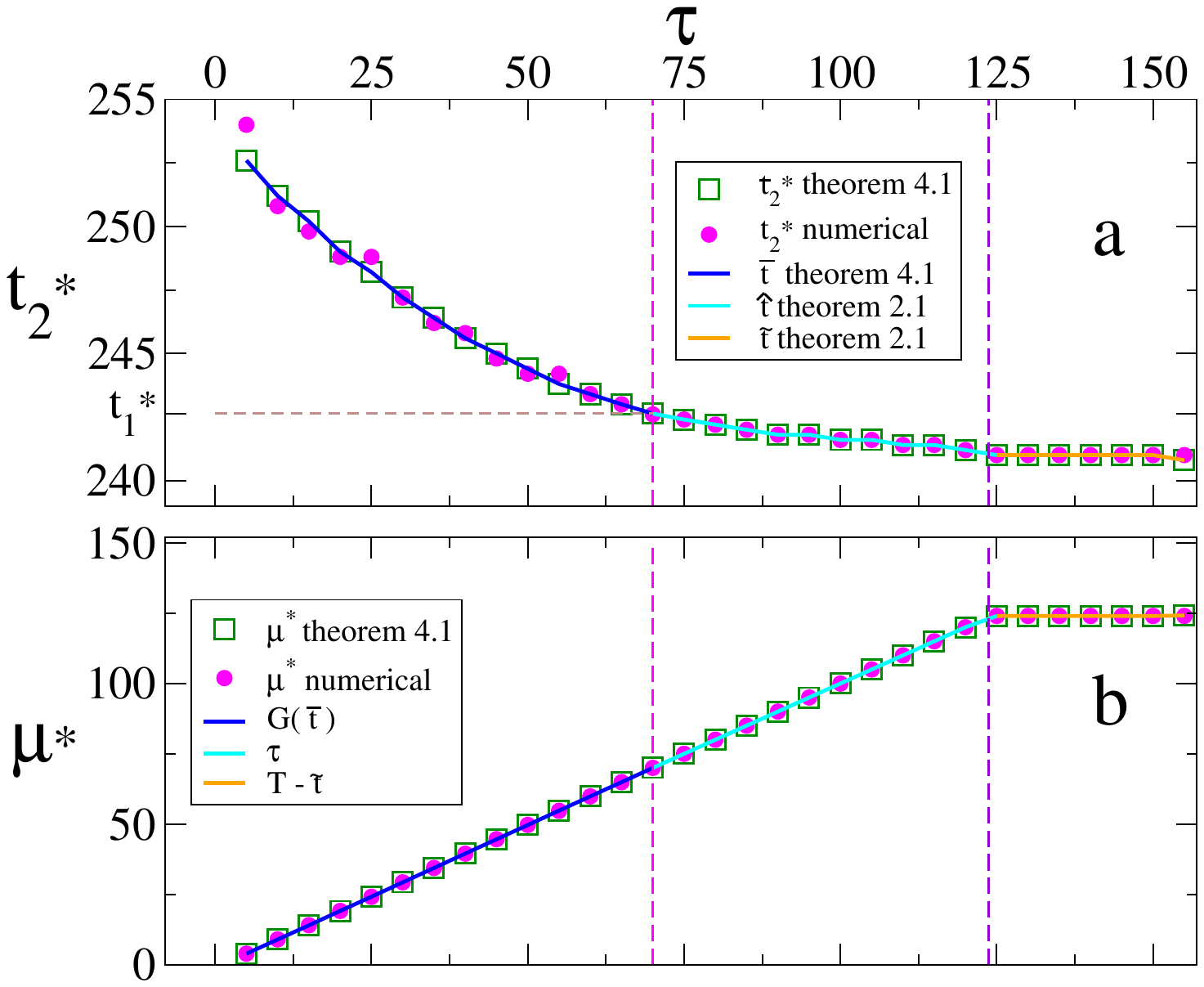} \\
    \vspace{-0.4cm}  
    \includegraphics[width=9cm]{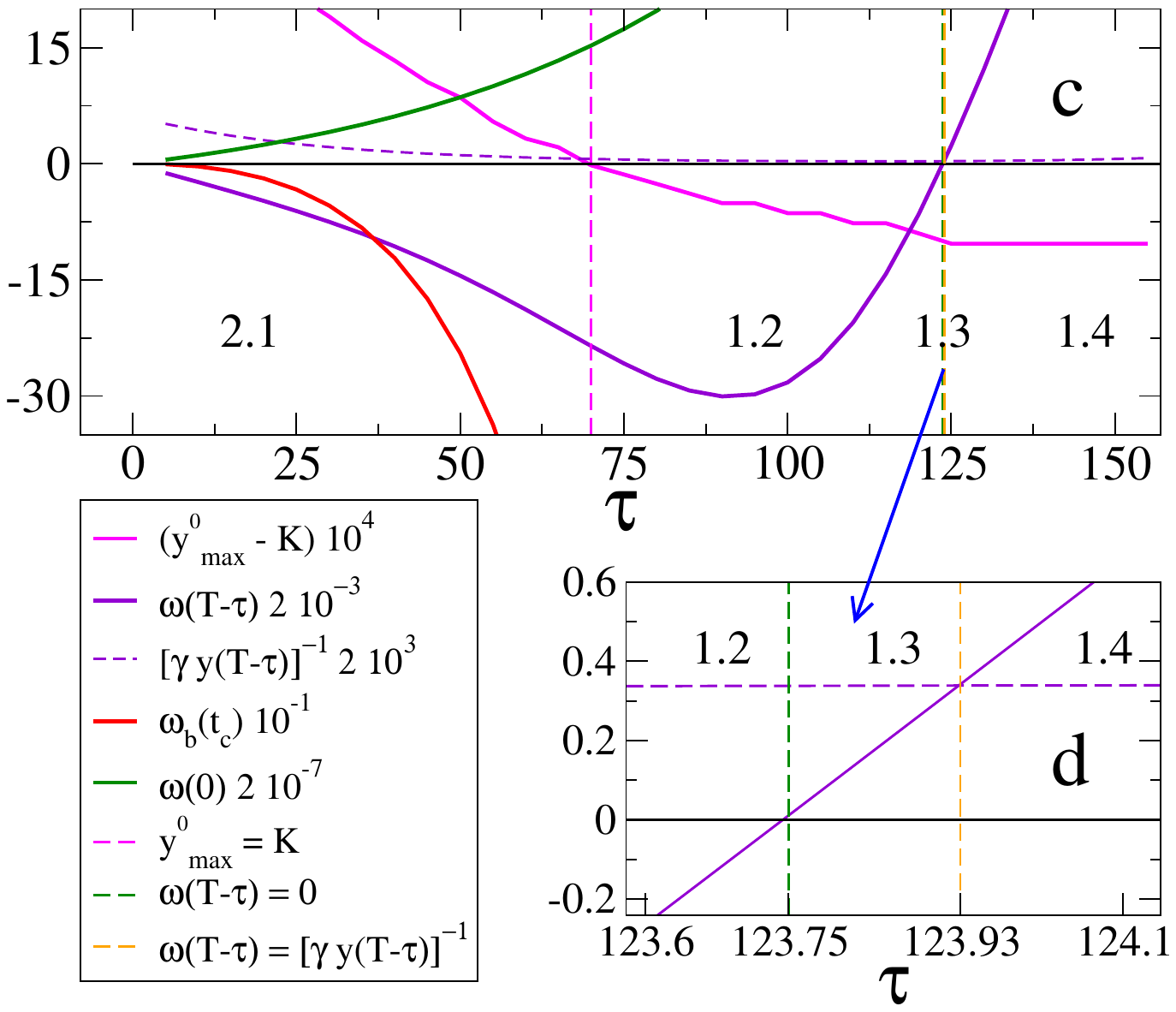}    
  \end{center}  
  \caption{Optimal solution for threshold $K=0.06$, control period $T=365$, recovery rate $\gamma=0.1$, strict quarantine reproduction number $\sigma_s=0.8$, and quarantine free reproduction number $\sigma_f=1.5$.  The initial condition is $x_0=1-10^{-1}$ and $y_0=10^{-6}$.  (a) Optimal initial time of the strict quarantine $t_2^*$ vs maximum allowed duration of the strict quarantine $\tau$.  (b) Optimal duration of the strict quarantine $\mu^*$ vs $\tau$.  In both panels (a) and (b), empty squares represent the exact solution (Theorem~\ref{te: control_optimo_sigma1gral_rest_v2}), solid circles correspond to the numerical solution that maximises $x_{\infty}$, and solid lines are the exact solution in each region.  (c) Graphical determination of the times $\tau \simeq 70$, $\tau \simeq 123.75$ and $\tau \simeq 123.93$ denoted by vertical dashed lines that define the four regions (cases) for the different behaviours of $t_2^*$ and $\mu^*$.  (d) Same as panel (c) with a smaller scale to see region $1.3$ in more detail.}
  \label{w-t2-mu-K-006}
\end{figure}

\begin{figure}
  \begin{center}
    \includegraphics[width=7.5cm]{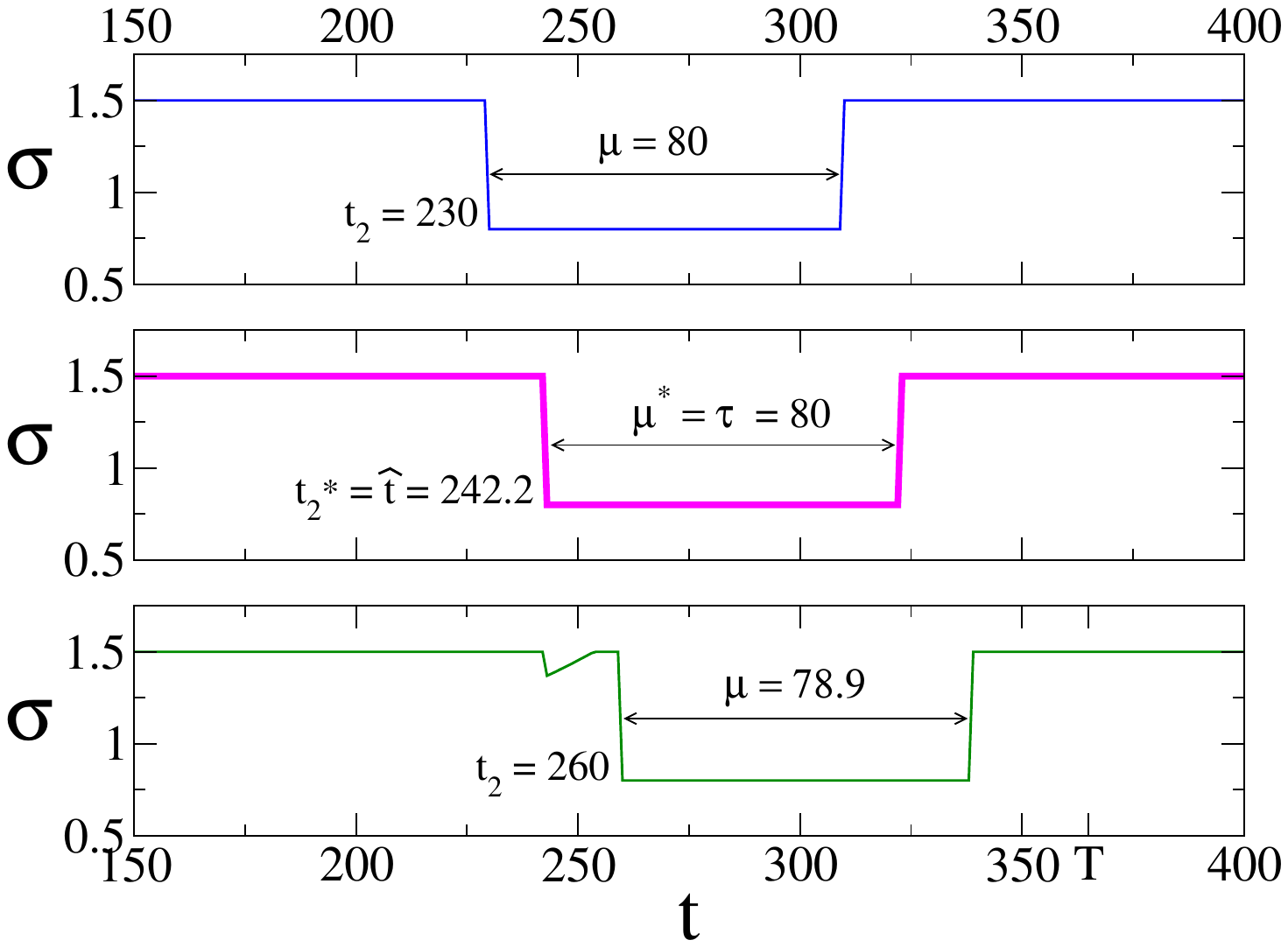} 
    \includegraphics[width=7.5cm]{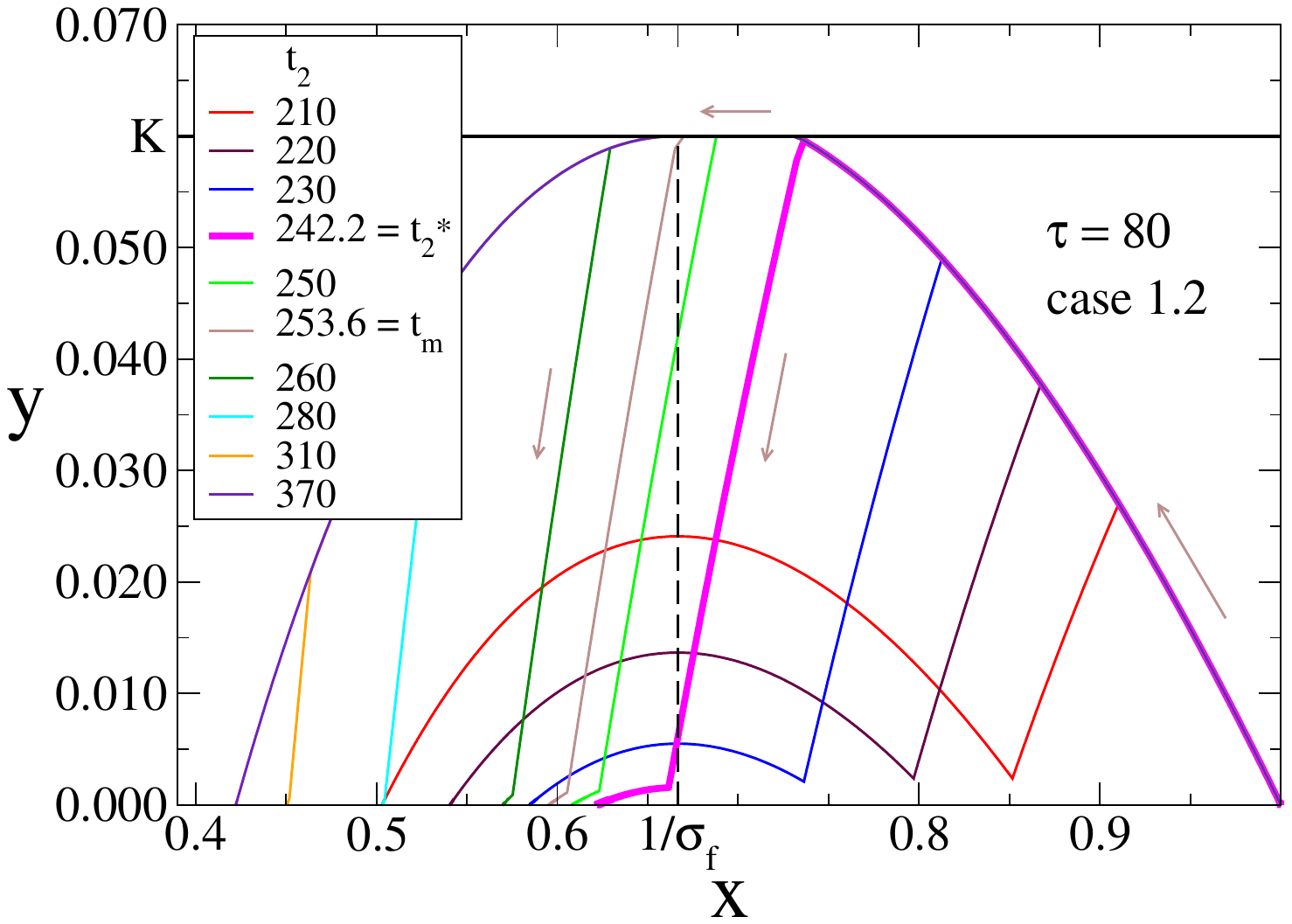} \\
    \vspace{-0.3cm}
    \includegraphics[width=7.5cm]{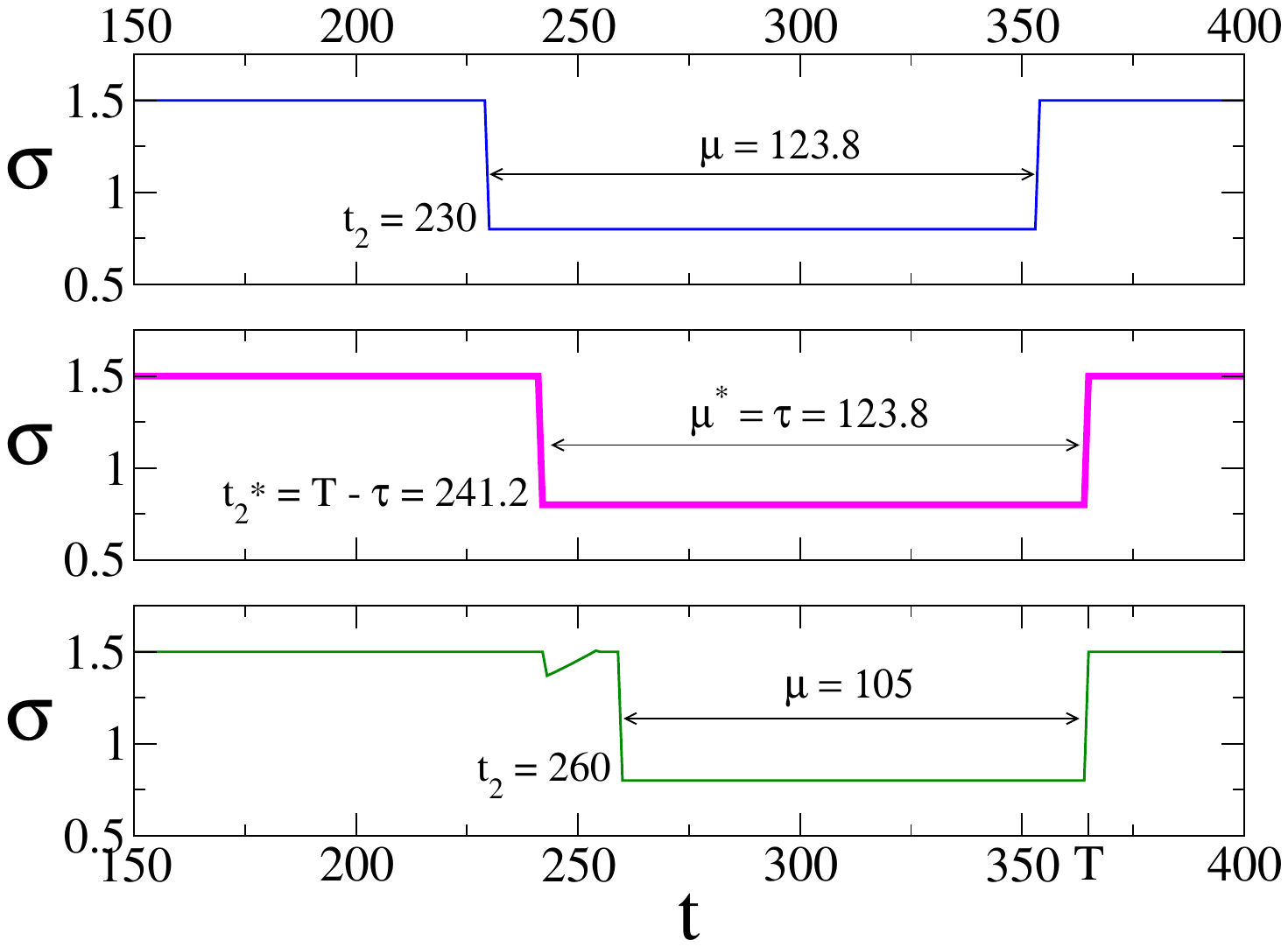} 
    \includegraphics[width=7.5cm]{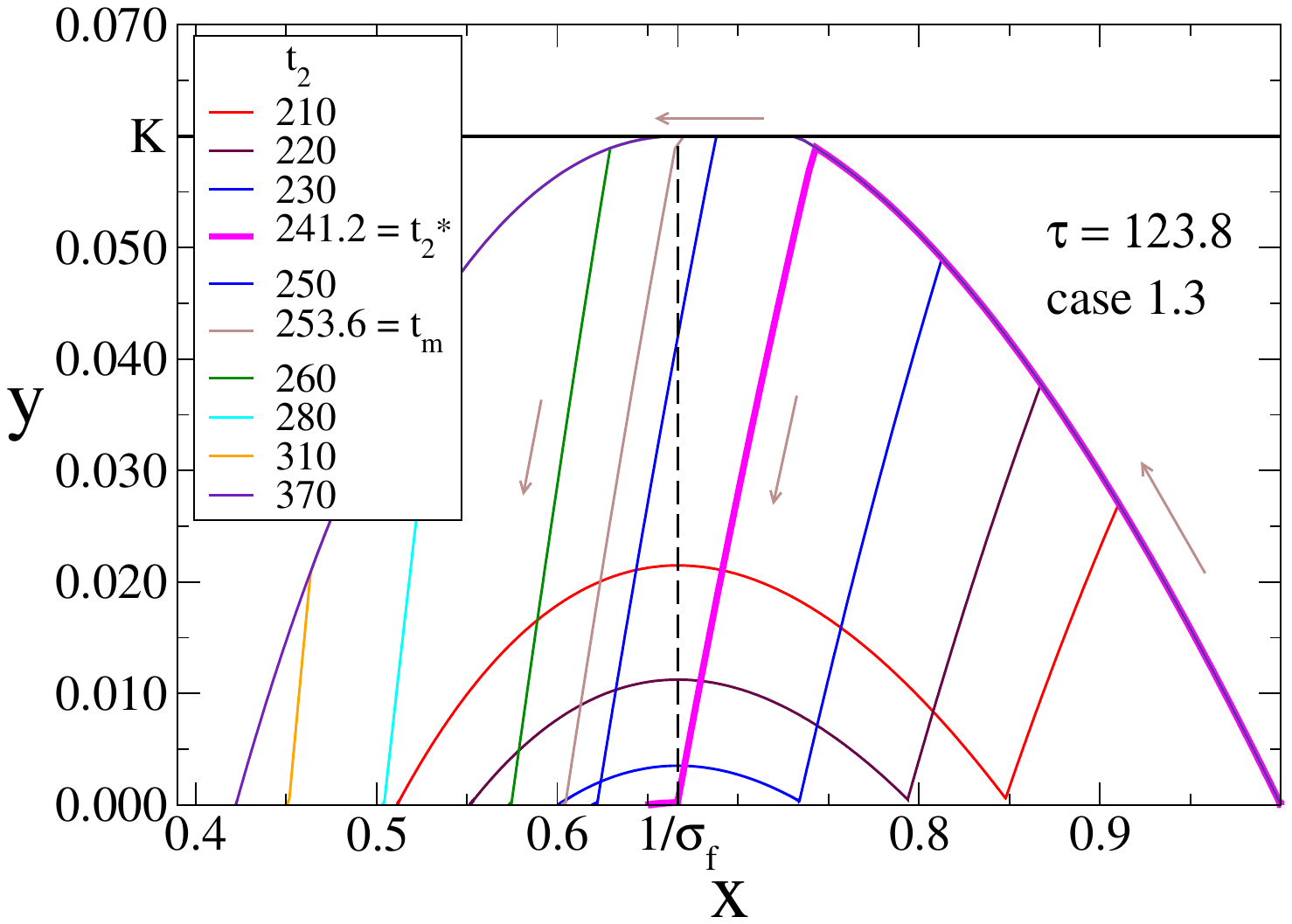} \\
    \vspace{-0.3cm}
    \includegraphics[width=7.5cm]{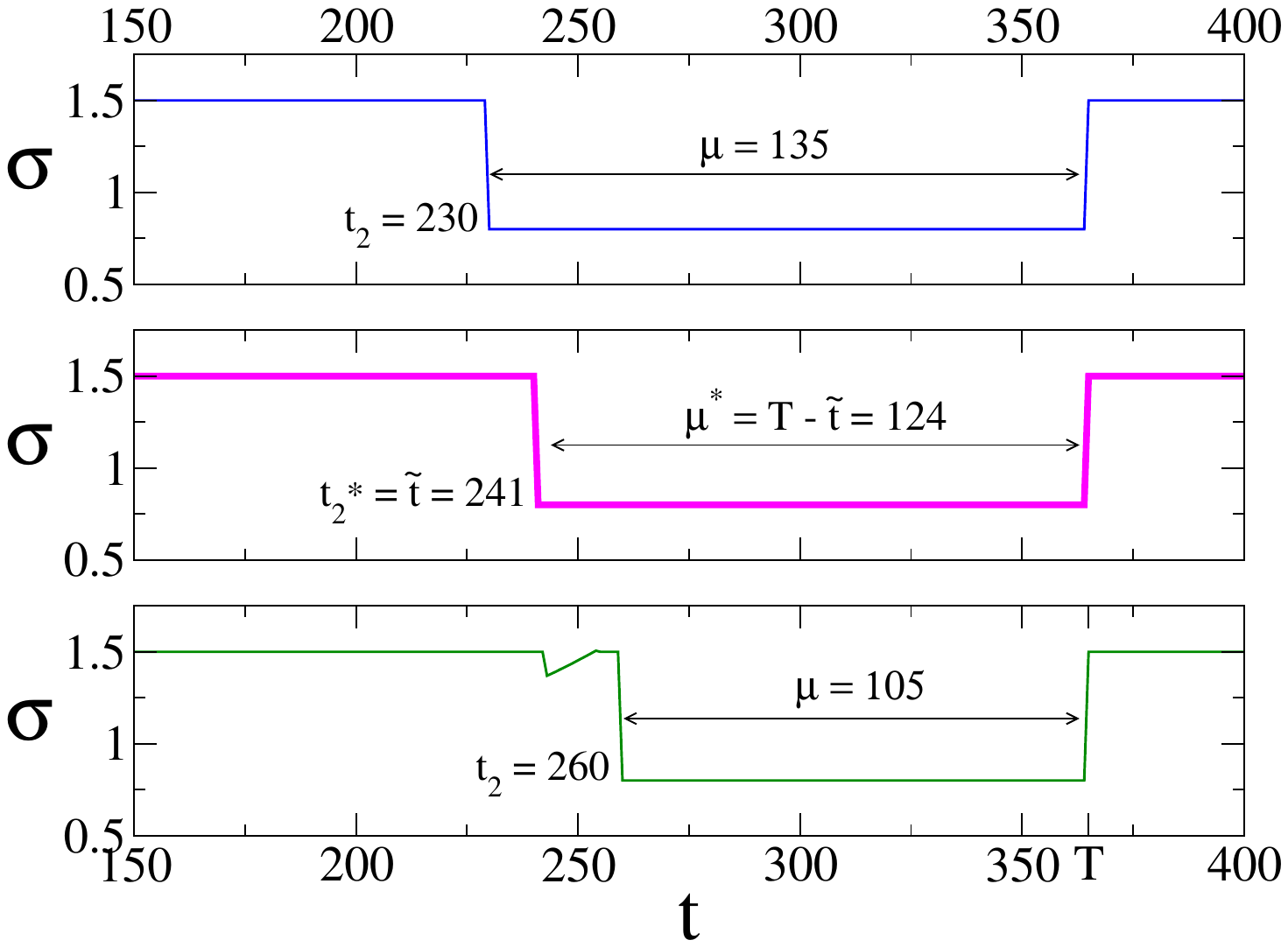} 
    \includegraphics[width=7.5cm]{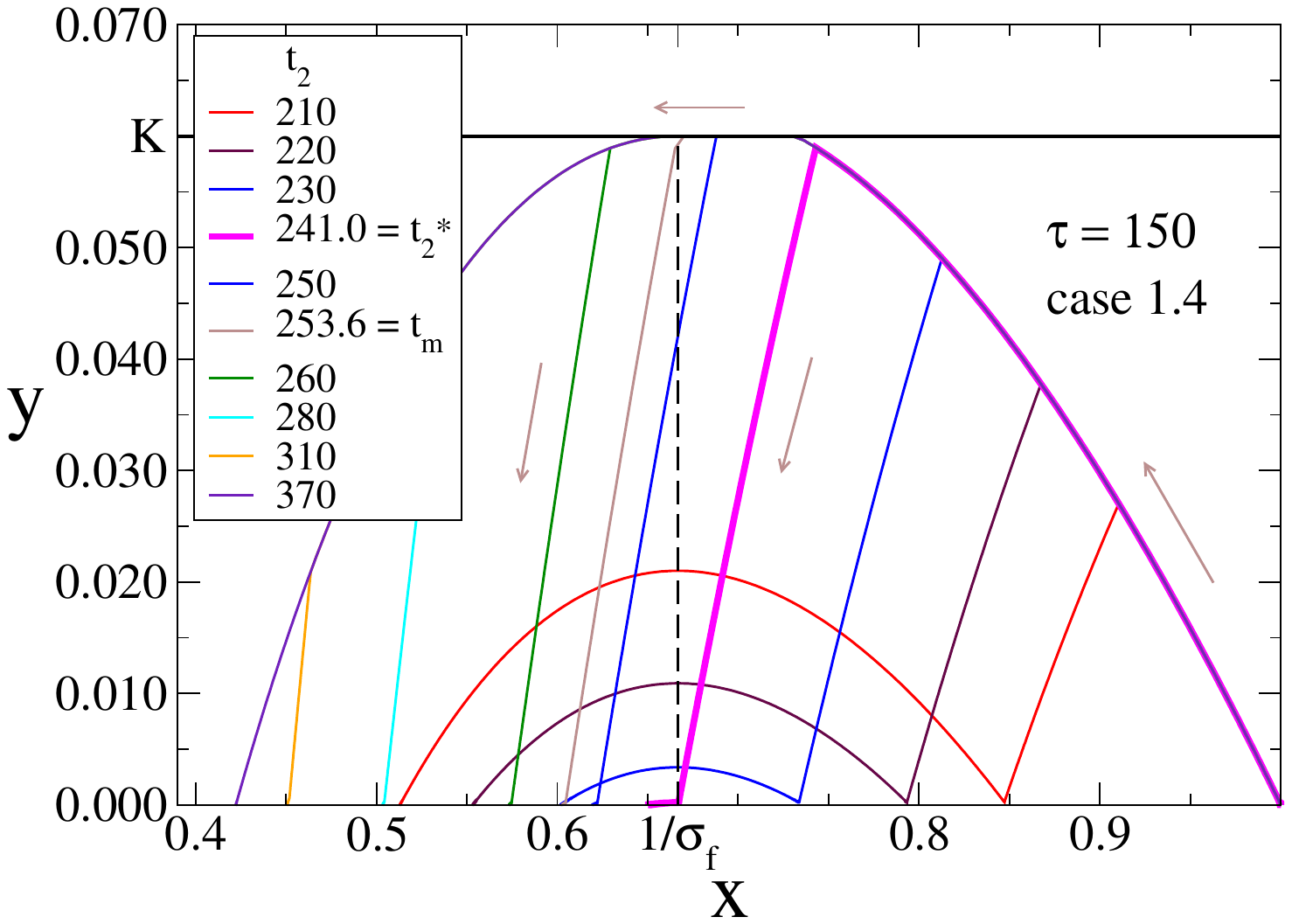} \\
    \vspace{-0.8cm}
  \end{center}    
  \caption{System's trajectory in the $x$--$y$ phase space (right
    panels) for $K=0.06$, $T=365$, $\gamma=0.1$, $\sigma_s=0.8$ and $\sigma_f=1.5$, and the three values of $\tau$ indicated in the legends, corresponding to three different cases of the optimum time $t_2^*$ (Theorem~\ref{te: control_optimo_sigma1gral_rest_v2}).  Each individual trajectory corresponds to a different initial time of the strict quarantine $t_2$.  Pink lines denote the optimum trajectory, where the strict quarantine starts at the optimum time $t_2^*$ before the hitting time $t_1^* \simeq 242.63$ for a period $\mu^*$ (minimum of $x_{\infty}$).  For the brown trajectory the strict quarantine starts when the system leaves the boundary arc $y(t)=K$ at $t_2=t_m \simeq 290.4$ ($x(t_m)=1/\sigma_f$).  Left panels show the time evolution of $\sigma(t)$ for three different initial times $t_2$ in each case.  The optimum times are $t_2^* \simeq 242.2$ for $\tau=80$ (top panels), $t_2^* \simeq 241.2$ for $\tau=123.8$ (middle panels), and $t_2^* \simeq 241.0$ for $\tau=150$ (bottom panels).}
  \label{sigma-x-y-K-006}
\end{figure}


\section{Summary and future work}
\label{se: summary}
In this work, we have analysed optimal quarantine measures using a time-dependent reproduction number within an SIR-type epidemic model. These control measures were considered as isolation measures, ranging from partial lockdown to non-intervention. 

The control strategy was derived based on necessary conditions provided by the Pontryagin Maximum Principle. This strategy is applied over a preset intervention period, aiming to minimise the final size of the epidemic while ensuring that medical care capacity and social and economic costs are not exceeded.

The optimal solution is a "wait, maintain, suspend, release" strategy (see Sereno \cite{Sereno2022} and Miclo \cite{Miclo2020}). It consists of at most four phases: an initial phase of no action with $\sigma^{*}(t)=\sigma_f$ that ends at time $t_1^*$, a second phase where the number of infected individuals is maintained at its maximum possible value $y^{*}(t)=K$ with $\sigma^{*}(t)=1/x^{*}(t)$ that ends at $t_2^*$, followed by a third phase of partial lockdown with $\sigma^{*}(t)=\sigma_s$ applied over a single interval that ends at $t_2^*+\mu^*$, and concluding with a final phase of no intervention where $\sigma^{*}(t)=\sigma_f$ until the control time $T$. Subsequently, a specific optimal policy is deduced by optimising the transition times $t_1^{*}$, $t_2^{*}$, and $t_2^{*}+\mu^{*}$ of this generic strategy.

As demonstrated in Theorem \ref{te: control_optimo_sigma1gral_rest_v2} and tested with numerical simulations, the optimal time $t_1^{*}$ to begin the maintain strategy depends on the initial values $(x_0,y_0)$ and the threshold $K$. The optimal time $t_2^{*}$ to finish the maintain strategy and begin the suspend phase also depends on the value $\tau$, which restricts the duration of both the maintain and suspend strategies.

Future work will focus on the consideration of more complex and realistic models. These models could combine different cost functionals from the literature (e.g., final size, $L^1$ norm for the control, peak prevalence) with more general restrictions on the control and final state.

\section{Declaration of Competing Interest}
The authors declares that they have no competing financial interests or personal relationships that could have appeared to influence the work reported in this paper.

\section{Supplement}

We begin by giving some results concerning the analysis of the function $F[t_1^*](t_2)$ defined in \eqref{eq: Ft1t2} in section \ref{se: caracterization_optimal}. 

\begin{lemma} \label{le: F_crece}
The function $F[t_1^*](t_2)$ defined in \eqref{eq: Ft1t2} is strictly decreasing. Moreover for $(\mu,t_2) \in {\mathcal R}[t_1^*]$ defined in \eqref{eq: regionR} it holds
\begin{align}
\label{eq: cota_mu_F}
\mu \le F[t_1^*](t_2) \le F[t_1^*](t_1^*) = \tau \, \text{ for all } \, t_2 \in [t_1^*,t_{f}].
\end{align}
\end{lemma}
\begin{proof}
We compute the derivative
\begin{align} \label{eq: derivada_Ft1}
\frac{d F[t_1^*]}{d t_2}(t_2)&= \frac{1}{\sigma_s-\sigma_f} \Big( \sigma_f  -  \frac{1}{x[t_1^*,t_2,\mu](t_2)}  \Big).
\end{align}
Using that $\sigma_s-\sigma_f <0$ and $ \sigma_f  -   \frac{1}{x[t_1^*,t_2,\mu](t_2)}  > 0$ (see Remark \ref{re: x2_y_sigmaf}), 
we obtain that $F$ is strictly decreasing. Therefore we obtain \eqref{eq: cota_mu_F}.
\end{proof}

\begin{lemma}\label{le: H_crece}
 Let $H:[t_{1}^{*},t_{f}] \to \mathbb{R}$ be the function defined by $H(t_2)=F[t_1^*](t_2)+t_2$. Then $H$ is an increasing function.
\end{lemma}

\begin{proof}
Computing the derivative of $H(t_2)$ we obtain
 \begin{align} \label{eq: deriv_H}
 H'(t_2)&= 1+\frac{d F[t_1^*]}{d t_2}(t_2)
= \frac{1}{\sigma_s-\sigma_f}\left(\sigma_s- \frac{1}{x[t_1^*,t_2,\mu](t_2)}\right) >0.
\end{align}
\end{proof}

Next we will do some preliminary calculations in order to compute the partial derivative $\frac{\partial J}{\partial \mu}(t_2,\mu)$ given by \eqref{eq: deriv_J_resp_mu_cuerpo} and also the derivative of $\tilde{J}(t_2)$ given by \eqref{eq: Jtilde_primat2} and its monotonicity.

We denote the solution of system
\begin{subequations} \label{eq: sistema_xy}
\begin{align}
&x'(r)= -\gamma \sigma(r) x(r) y(r),   \quad  x(t)=x \\
 &y'(r) = \gamma \sigma(r) x(r) y(r) - \gamma y(r),  \quad y(t)=y
 \end{align}
 \end{subequations}
  for $r\ge t$ by 
\begin{align} \label{eq: psi} \Psi(r,t,x,y,\sigma), \,\, & \text{ with } \sigma\equiv \sigma_s, \sigma\equiv \sigma_f \, \text{ or }\, \sigma(r)=\sigma_b(r)=1/x(r) \nonumber \\  &\text{ and initial data } \,\,    (x,y)\in {\cal D} \,\, \text{ at time } \, t.  \end{align}
If we denote  $(x_1,y_1)=\Psi(t^*_1,t_0,x_0,y_0,\sigma_f) \label{eq: x1_y1}$, $(x_2,y_2)=\Psi(t_2,t^*_1,x_1,y_1,\sigma_b)$, $(x_3,y_3)=\Psi(t_2+\mu,t_2,x_2,y_2,\sigma_s)$, then, from \eqref{eq: sigma opt gral_kappa0}
\begin{equation} \label{eq: xy_t_eta}
(x[{t_1^*,t_2,\mu}](r),y[{t_1^*,t_2,\mu}](r))= \left\{ \begin{array}{ll}
\Psi(r,t_0,x_0,y_0,\sigma_f)& \text{for } 0\le r \le t^*_1 \\
\Psi(r,t^*_1,x_1,y_1,\sigma_b) & \text{ for } t^*_1 < r \le t_2\\
\Psi(r,t_2,x_2,y_2,\sigma_s)& \text{ for } t_2< r \le t_2+\mu \\
\Psi(r,t_2+\mu,x_3,y_3,\sigma_f)& \text{ for } t_2+\mu< r \le T
\end{array}\right. .
\end{equation}

While the majority of the subsequent content is covered in the Supplement of \cite{Balderrama22}, we will incorporate certain proofs essential to ensure the self-sufficiency of this article.

We recall two properties for the solutions of ordinary differential equations. First, the relation
between the derivative with respect to initial time and the derivatives with respect to initial
data give us the equation
\begin{align} \label{eq: deriv_sol_resp_tiempo_inicial}
\frac{\partial \Psi_j(s,t,x,y,\sigma)}{\partial t}&=\frac{\partial \Psi_j(s,t,x,y,\sigma)}{\partial x} \gamma \sigma x y - \frac{\partial \Psi_j(s,t,x,y,\sigma)}{\partial y} \gamma y(\sigma x -1)
\end{align}
for $\Psi$ defined in \eqref{eq: psi}, with $s\ge t$, $\sigma \in \left\{ \sigma_s,\sigma_f\right\}$, initial data $(x,y)\in {\cal D}$ at time $t$ and $j=1,2$.

Second, the dependence of the solution $\Psi(s,t,x,y,\sigma)$ with respect to initial data
$x,y$ is given by the following known equations. For simplicity of notation, when there is no
risk of confusion, we will denote $\Psi(s)$ for $\Psi(s,t,x,y,\sigma)$,

\begin{equation}
\left( \begin{array}{ll}
\frac{\partial \Psi_1}{\partial x} & \frac{\partial \Psi_1}{\partial y} \\
\frac{\partial \Psi_2}{\partial x}&\frac{\partial \Psi_2}{\partial y}
\end{array}\right)' (s)=
\left( \begin{array}{ll}
-\gamma \sigma \Psi_2 (s) & -\gamma \sigma \Psi_1(s)\\
\gamma \sigma \Psi_2(s)&\gamma ( \sigma \Psi_1(s)-1)
\end{array}\right) .
\left( \begin{array}{ll}
\frac{\partial \Psi_1}{\partial x} & \frac{\partial \Psi_1}{\partial y} \\
\frac{\partial \Psi_2}{\partial x}&\frac{\partial \Psi_2}{\partial y}
\end{array}\right)(s) .
\end{equation}
with initial data
\begin{equation}
\left( \begin{array}{ll}
\frac{\partial \Psi_1}{\partial x} & \frac{\partial \Psi_1}{\partial y} \\
\frac{\partial \Psi_2}{\partial x}&\frac{\partial \Psi_2}{\partial y}
\end{array}\right)(t)=Id .
\end{equation}
Then, denoting $t_3=t_2+\mu$, $t_4=T$, $\sigma_2=\sigma_s$ and $\sigma_3=\sigma_f$, for $i=2,3$, we define
\begin{subequations}  \label{eq: def_u_v}
\begin{align}
u_i(s)=u(s,t_i,x_i,y_i,\sigma_i)=\frac{\partial \Psi_1}{\partial x_i}(s,t_i,x_i,y_i,\sigma_i) - \frac{\partial \Psi_1}{\partial y_i}(s,t_i,x_i,y_i,\sigma_i),  \label{eq: def_u} \\
v_i(s)=v(s,t_i,x_i,y_i,\sigma_i)=\frac{\partial \Psi_2}{\partial x_i}(s,t_i,x_i,y_i,\sigma_i) - \frac{\partial \Psi_2}{\partial y_i}(s,t_i,x_i,y_i,\sigma_i),  \label{eq: def_v}
\end{align}
\end{subequations}
for $s\in [t_i,t_{i+1}]$,
that satisfy the system of equations on $u_i$ and $v_i$
\begin{equation} \label{eq: derivparc_u_v}
\left( \begin{array}{l}
u_i'(s) \\
v_i'(s)
\end{array}\right)=
\left( \begin{array}{ll}
-\gamma \sigma_i \Psi_2(s) & -\gamma \sigma_i \Psi_1(s)\\
\gamma \sigma_i \Psi_2(s)&\gamma ( \sigma_i \Psi_1(s)-1)
\end{array}\right) .
\left( \begin{array}{l}
u_i(s) \\
v_i(s)
\end{array}\right) .
\end{equation}
with initial data
\begin{equation}
\left( \begin{array}{ll}
u_i(t_i)\\
v_i(t_i)
\end{array}\right)=
\left( \begin{array}{ll}
1\\
-1
\end{array}\right).
\end{equation}

Moreover, using that for any $(x_i,y_i)\in {\cal D}$, $\Psi(s,t_i,x_i,y_i,\sigma_i)$, satisfies for $s\in [t_i,t_{i+1}]$
\begin{align*}
\Psi_1(s,t_i,x_i,y_i,\sigma_i) e^{-\sigma_i (\Psi_1(s,t_i,x_i,y_i,\sigma_i)+\Psi_2(s,t_i,x_i,y_i,\sigma_i))} =x_i e^{-\sigma_i (x_i+y_i)},
\end{align*}
we compute the derivatives with respect to $x_i$ and $y_i$  for $s\in [t_i,t_{i+1}]$
\begin{subequations} \label{eq: deriv_Psi_12}
 \begin{align}
\frac{\partial \Psi_1}{\partial x_i}(s)-\sigma_i \Psi_1(s) \left( \frac{\partial \Psi_1}{\partial x_i}(s)+\frac{\partial \Psi_2}{\partial x_i}(s)\right)  &=(1-\sigma_i x_i) \frac{\Psi_1(s)}{x_i}, \label{eq: deriv_Psi_1}  \\
\frac{\partial \Psi_1}{\partial y_i}(s)-\sigma_i \Psi_1(s) \left( \frac{\partial \Psi_1}{\partial y_i}(s)+\frac{\partial \Psi_2}{\partial y_i}(s)\right) &=-\sigma_i \Psi_1(s). \label{eq: deriv_Psi_2}
\end{align}
\end{subequations}
where, for simplicity of notation, we have called $\Psi_j(s)=\Psi_j(s,t_i,x_i,y_i,\sigma_i)$ for $j=1,2$ and $s\in [t_i,t_{i+1}]$.
Then, subtracting the last two equations
\begin{align} \label{eq: expres_uts}
 u_i(s) -\sigma_i  \Psi_1(s,t_i,x_i,y_i,\sigma_i) (  u_i(s)+ v_i(s))=\frac{ \Psi_1(s,t_i,x_i,y_i,\sigma_i)}{x_i}
\end{align}
for $s\in[t_i,t_{i+1}]$ and therefore using \eqref{eq: expres_uts}, from \eqref{eq:
derivparc_u_v} we have that $u_i$ satisfies the ordinary differential equation
{\small \begin{align*}
 u_i'(s)&=\gamma \left(\sigma_i \Psi_1(s,t_i,x_i,y_i,\sigma_i)-\sigma_i \Psi_2(s,t_i,x_i,y_i,\sigma_i) -1\right) u_i(s)+\gamma  \frac{\Psi_1(s,t_i,x_i,y_i,\sigma_i)}{x_i}, \\
u_i(t_i)&= 1.
\end{align*}}

Thus,  for $s\in[t_i,t_{i+1}]$ when $i=2,3$, we obtain
\begin{align}
u_i(s)&=\frac{\Psi_1(s)\Psi_2(s)}{x_i y_i}+\gamma \frac{\Psi_1(s)\Psi_2(s)}{x_i} \int_{t_i}^{s} \frac{1}{\Psi_2(r)}dr \nonumber\\
& = \frac{\Psi_1(s)}{x_i }+\gamma \sigma_i \frac{\Psi_1(s)\Psi_2(s)}{x_i} \int_{t_i}^{s} \frac{\Psi_1(r)}{\Psi_2(r)}dr,  \label{eq: u_ti} \\
v_i(s)&=-\frac{\Psi_1(s)}{x_i}+\frac{(1-\sigma_i \Psi_1(s))\Psi_2(s)}{x_i}\gamma  \int_{t_i}^{s} \frac{\Psi_1(r)}{\Psi_2(r)}dr, \label{eq: v_ti} \\
u_i(s)+v_i(s)&=\frac{\Psi_2(s)}{x_i}\gamma  \int_{t_i}^{s} \frac{\Psi_1(r)}{\Psi_2(r)}dr. \label{eq: umasv_ti}
\end{align}


From  \eqref{eq: xy_t_eta}, using \eqref{eq: deriv_sol_resp_tiempo_inicial} and \eqref{eq: def_u_v}, after some computations we have, for $r\in (t_2+\mu,T]$
\begin{subequations} \label{eq: deriv_x_y_mu}
\begin{align}
\frac{\partial x}{\partial \mu} [t_1^*,t_2,\mu](r) 
&= \gamma  x_3y_3 (\sigma_f - \sigma_s) u_3(r) \\
\frac{\partial y}{\partial \mu} [t_1^*,t_2,\mu](r) &= 
\gamma  x_3y_3 (\sigma_f - \sigma_s) v_3(r)
\end{align}
\end{subequations}
and for $r\in [0,t_2+\mu]$
\begin{align} \label{eq: deriv_x_mu_0_t2masmu}
 \frac{\partial x[{t_1^*,t_2,\mu}]}{\partial \mu}(r)=0, \,\, \frac{\partial y[{t_1^*,t_2,\mu}]}{\partial \mu}(r)=0.
\end{align}

Therefore, combining \eqref{eq: deriv_xinfty} with \eqref{eq: deriv_x_y_mu} for $r=T$ and \eqref{eq: expres_uts} for $i=3$ we obtain
\begin{align}  \label{eq: deriv_J_resp_mu}
\frac{\partial J}{\partial \mu}(t_2, \mu) &=\gamma  y_3 (\sigma_f - \sigma_s)
\frac{x_{\infty}[t_1^*,t_2,\mu]}{1- \sigma x_{\infty}[t_1^*,t_2,\mu]}
\end{align}

where $x_{\infty}[t_1^*,t_2,\mu]=x_{\infty}(x[t_1^*,t_2,\mu)](T),y[t_1^*,t_2,\mu](T),\sigma_f)$.

In what follows we will compute $\tilde{J}'(t_2)$. We begin by computing 
 the derivative of $x[t_1^*,t_2,\mu](r)$ with respect to $t_2$. 
From  \eqref{eq: xy_t_eta}, we have that for $r\in[0,t_2)$, 
\begin{align*}
\frac{\partial x}{\partial t_2}[t_1^*,t_2, \mu](r) = 0, \,\, \frac{\partial x}{\partial t_2}[t_1^*,t_2, \mu](r) = 0.
\end{align*}

For $r\in (t_2,t_2+\mu)$, using \eqref{eq: deriv_sol_resp_tiempo_inicial}, \eqref{eq: def_u_v} and the fact that $x_2=x[t_1^*,t_2,\mu](t_1^*)-\gamma K (t_2-t_1)$ and $y_2=K$, we obtain
\begin{subequations} \label{eq: deriv_x_y_t2_t2_t2masmu}
\begin{align}
\frac{\partial x}{\partial t_2} [t_1^*,t_2,\mu](r)
&= \gamma   K \left( \sigma_s x_2  -1\right) u_2(r) \label{eq: deriv_x_t2_t2_t2masmu} \\
\frac{\partial y}{\partial t_2} [t_1^*,t_2,\mu](r) 
&= \gamma  K \left( \sigma_s x_2  -1\right) v_2(r) \label{eq: deriv_y_t2_t2_t2masmu}
\end{align}
\end{subequations}

Finally, for $r\in (t_2+\mu,T]$, from  \eqref{eq: xy_t_eta}, using \eqref{eq: sistema_xy} and \eqref{eq: psi} with $\sigma=\sigma_s$ and initial data $(x_2,y_2)$ at time $t_2$, 
\eqref{eq: deriv_sol_resp_tiempo_inicial} and \eqref{eq: def_u_v} we obtain
\begin{subequations} \label{eq: deriv_x3_y3}
\begin{align}
\frac{\partial x_3}{\partial t_2}   \label{eq: deriv_x3}
&= -\gamma \sigma_s x_3 y_3 + \gamma K \left( \sigma_s x_2  -1\right) u_2(t_2+\mu)\\
\frac{\partial y_3}{\partial t_2}   \label{eq: deriv_y3}
&=\gamma y_3 (\sigma_s x_3 -1)+ \gamma K \left( \sigma_s x_2  -1\right) v_2(t_2+\mu).
\end{align}
\end{subequations}

Thus, 
{\small \begin{align*}
\frac{\partial x}{\partial t_2} [t_1^*,t_2,\mu](r) &= \gamma  x_3y_3 (\sigma_f - \sigma_s) u_3(r)+\gamma K\left(\sigma_s x_2-1\right) \cdot \\ 
& \left[ \frac{\partial \Psi_1}{\partial x_3}(r, t_2+\mu, x_3,y_3,\sigma_f) u_2(t_2+\mu)+\frac{\partial \Psi_1}{\partial y_3}(r, t_2+\mu, x_3,y_3,\sigma_f)v_2(t_2+\mu)  \right]
\end{align*}}
and
{\small \begin{align*}
\frac{\partial y}{\partial t_2} [t_1^*,t_2,\mu](r) &= 
 \gamma  x_3y_3 (\sigma_f - \sigma_s) v_3(r)+\gamma K\left(\sigma_s x_2-1\right) \\
 & \left[ \frac{\partial \Psi_2}{\partial x_3}(r, t_2+\mu, x_3,y_3,\sigma_f) u_2(t_2+\mu)+\frac{\partial \Psi_2}{\partial y_3}(r, t_2+\mu, x_3,y_3,\sigma_f)v_2(t_2+\mu)  \right].
\end{align*}}
After some computations, combining \eqref{eq: deriv_xinfty} with \eqref{eq: deriv_Psi_12} and \eqref{eq: expres_uts} for $s=T$ and for $s=t_3$ we obtain

{\small \begin{align*} 
\frac{\partial J}{\partial t_2}(t_2, \mu)
&=  \frac{x_{\infty}[t_1^*,t_2,\mu]}{(1-\sigma_f x_{\infty}[t_1^*,t_2,\mu])} \cdot  \nonumber\\
& \hspace{0.6cm} \left[\gamma (\sigma_f-\sigma_s) y_3 + \gamma K (\sigma_sx_2-1)
\left( \frac{1}{x_2}- (\sigma_f-\sigma_s)(u_2(t_2+\mu)+v_2(t_2+\mu))\right)\right] .
\end{align*}}
From \eqref{eq: umasv_ti} and using that integrating $\left( \frac{1}{y[t_1^*,t_2,\mu]}\right)' (r)$ on the interval $[t_2,t_3]$  (see equations (74) and (75) in \cite{Balderrama22} for $i=2$ and $s=t_3$) we obtain 
\begin{align} \label{eq: cuenta_aux}
1+\gamma y_3 (\sigma_s-\sigma_f) \int_{t_2}^{t_2+\mu} \frac{x[t_1^*,t_2,\mu](r)}{y[t_1^*,t_2,\mu](r)} dr = \frac{y_3}{K}-\gamma y_3 \int_{t_2}^{t_2+\mu} \frac{\sigma_f x[t_1^*,t_2,\mu](r)-1}{y[t_1^*,t_2,\mu](r)} dr,
\end{align}
 we conclude that
{\small\begin{align} \label{eq: deriv_J_resp_t2}
\frac{\partial J}{\partial t_2}(t_2, \mu)
&=   \frac{\gamma y_3 x_{\infty}[t_1^*,t_2,\mu]}{(1-\sigma_f x_{\infty}[t_1^*,t_2,\mu])}  \left[ \sigma_f-\frac{1}{x_2}-\gamma K \left(\sigma_s-\frac{1}{x_2}\right)  \int_{t_2}^{t_2+\mu} \frac{\sigma_f x[t_1^*,t_2,\mu](r)-1}{y[t_1^*,t_2,\mu](r)} dr\right].	
\end{align}}

Therefore, 
 from \eqref{eq: deriv_J_resp_t2}, \eqref{eq: deriv_J_resp_mu} and \eqref{eq: derivada_Ft1}  we obtain for 
 $t_2 \in [t_1^*,t_{c})$
{\small  \begin{align} \label{eq: deriv_J_resp_t2_parte1}
\frac{\partial}{\partial t_2}J(t_2,F[t_1^*](t_2)) &=\frac{\gamma^2 y_3 K (1-\sigma_s x_2) x_{\infty}[t_1^*,t_2,F[t_1^*](t_2)]}{x_2 (1-\sigma_f x_{\infty}[t_1^*,t_2,F[t_1^*](t_2)])} \int_{t_2}^{t_2+F[t_1^*](t_2)}\frac{\sigma_f x[t_1^*,t_2,F[t_1^*](t_2)]](r)-1}{y[t_1^*,t_2,F[t_1^*](t_2)]](r)}dr,
 \end{align}}
and for $t_2 \in (t_{c},t_f]$
{\small \begin{align}  \label{eq: deriv_J_resp_t2_parte2}
\frac{\partial}{\partial t_2}J(t_2,T-t_2) &= \frac{\gamma^2 y_3 K(1-\sigma_s x_2) x_{\infty}[t_1^*,t_2,T-t_2] }{x_2 (1-\sigma_f x_\infty[t_1^*,t_2,T-t_2])}\left( \int_{t_2}^{T} \frac{\sigma_f x[t_1^*,t_2,T-t_2](r)-1}{y[t_1^*,t_2,T-t_2](r)} dr -\frac{1}{\gamma K} \right).
 \end{align}}

\begin{lemma} \label{le: varphi_decrece}
The function $\varphi:[t_1^*,t_{c}] \to \mathbb{R}$
$$
\varphi\left(t_{2}\right)=\Psi_{1} (t_{2}+F[t_1^{*}] (t_{2}), t_{2}, x_{2}, y_{2}, \sigma_{s})
$$
is decreasing.
\end{lemma}

\begin{proof}
Using similar computations to \eqref{eq: deriv_x3} combined with \eqref{eq: deriv_H} we obtain
\begin{align*}
 \varphi^{\prime}\left(t_{2}\right) 
  & =(\sigma_{s} x_2-  1 )  \left(  \frac{\gamma \sigma_{s} x_{3} y_{3}}{x_2} 
 +\gamma  K  u_{2} (t_{2}+F[t_{1}^{*}] (t_{2})) \right),
 \end{align*}
 concluding the desired result from the fact that $\sigma_s x_2<1$ and $u_2(t_2+F[t_{1}^{*}] (t_{2}))>0$ (see \eqref{eq: u_ti}).
\end{proof}

\bibliographystyle{plain}

\bibliography{references_sir_hospi}

\end{document}